\documentclass[12pt]{amsart}
\usepackage{amsmath, amsthm, amscd, amsfonts, amssymb, graphicx, xcolor}
\usepackage{hyperref}

\usepackage{mathtools}
\usepackage{enumitem} 

\usepackage{tikz-cd} 

\usepackage{graphicx}
\graphicspath{ {images/} }

\newcommand{\N}{\ensuremath{\mathbb{N}}}

\newcommand{\R}{\ensuremath{\mathbb{R}}}
\newcommand{\C}{\ensuremath{\mathbb{C}}}
\newcommand{\D}{\ensuremath{\mathbb{D}}}
\newcommand{\Sone}{\ensuremath{\mathbb{S}^{1}}}

\newcommand{\la}{\ensuremath{\lambda}}
\newcommand{\eps}{\ensuremath{\varepsilon}}

\DeclareMathOperator\capacity{Cap}

\newtheorem{theorem}{Theorem}[section]
\newtheorem{lemma}[theorem]{Lemma}
\newtheorem{proposition}[theorem]{Proposition}
\newtheorem{corollary}[theorem]{Corollary}
\theoremstyle{definition}
\newtheorem{definition}[theorem]{Definition}

\newtheorem{conjecture}[theorem]{Conjecture}

\theoremstyle{remark}
\newtheorem{remark}[theorem]{Remark}
\numberwithin{equation}{section}

\begin{document}
\setcounter{page}{1}

\color{darkgray}{
\noindent 

\centerline{}

\centerline{}

\title[Flexible curves and Hausdorff dimension]{Flexible curves and Hausdorff dimension}

\author[Alex Rodriguez]{Alex Rodriguez}
\address{Department of Mathematics, Stony Brook University, New York, USA.\\
	\textsc{\newline \indent 
	   \href{https://orcid.org/0000-0001-9097-4025%
	     }{\includegraphics[width=1em,height=1em]{orcid2} {\normalfont https://orcid.org/0000-0001-9097-4025}}
	       }}}
\email{\textcolor[rgb]{0.00,0.00,0.84}{alex.rodriguez@stonybrook.edu}}



\date{January 20, 2026.
\newline \indent The author is partially supported by NSF grant DMS 2303987.} 

\begin{abstract}
We show that given a log-singular circle homeomorphism $h$ and given any $s\in[1,2]$, there is a flexible curve of Hausdorff dimension $s$ with welding $h$. We also see that there is another curve with welding $h$ and positive area. In particular, this implies that given a flexible curve $\Gamma$, there is a homeomorphism of the plane $\phi\colon\C\to\C$, conformal off $\Gamma$, so that $\phi(\Gamma)$ has positive area. This answers a particular case of the corresponding conjecture for general non-conformally removable sets, for a class of curves that is residual in the space of all Jordan curves.
\end{abstract} 

\maketitle


\section{Introduction}

Let $\gamma\subset\C_{\infty}$ be a Jordan curve and consider $\Omega, \Omega^{*}$ the two complementary components of $\gamma$ in $\C_{\infty}$, the Riemann sphere. By the Riemann mapping theorem there are conformal maps $f\colon\D\to\Omega$ and $g\colon\C_{\infty}\setminus\overline{\D}=\D^{*}\to\Omega^{*}$, as represented in Figure \ref{figure:welding_def}. By the Carath\'eodory--Torhorst theorem, these conformal maps extend to be homeomorphisms of the unit circle to $\gamma$. This yields the orientation-preserving homeomorphism $h=g^{-1}\circ f\colon\Sone\to\Sone$ which we call a \textit{conformal welding} (or \textit{welding} for short). The welding $h\colon\Sone\to\Sone$ is determined by $\gamma$ up to pre- and post- composition by automorphisms of $\D$, we denote this class of circle homeomorphisms by $[h]$. Therefore, the \textit{conformal welding correspondence} $\mathcal{W}\colon[\gamma]\to [h]$
is well-defined, where $[\gamma]$ indicates the curves that are M\"obius images of $\gamma$. This correspondence is not surjective, for instance Oikawa \cite[Example 1]{MR125956} provided a family of circle homeomorphisms that are not weldings. In this paper we investigate its non-injectivity for log-singular circle homeomorphisms $h\colon\Sone\to\Sone$, that is, there is a zero logarithmic capacity Borel set $E\subset\Sone$ so that $h(\Sone\setminus E)$ has zero logarithmic capacity. 

\begin{theorem}\label{theorem:non-injectivity}
Let $h$ be a log-singular circle homeomorphism. Then $h$ is the conformal welding of a positive area curve, and of a curve of Hausdorff dimension $s$ for any $s\in[1,2]$.
\end{theorem}

Log-singular circle homeomorphisms are highly irregular, but are still weldings. This was first proved by Bishop \cite{ChrisWeldingAnnals} and it also follows from Theorem \ref{theorem:non-injectivity}. On the other side, well-behaved self-maps of the circle, like quasisymmetric maps, are weldings. This is the so called \textit{fundamental theorem of conformal welding} and it was first proven by Pfluger in \cite{Pfluger} by using the measurable Riemann mapping theorem. Shortly after, Lehto and Virtanen \cite{MR125962} gave a different proof, also by using quasiconformal mappings. Some work from Astala, Jones, Kupiainen and Saksman \cite{JonesActa2011} proves that regularity \textit{at most scales} is a sufficient condition for a homeomorphism to be a welding. However, we still lack a criteria that characterizes more general classes of weldings. Recently, I proved that every orientation-preserving circle homeomorphism is a composition of two conformal weldings \cite{Alex:Welding}. For a survey on conformal welding, the reader is encouraged to read Hamilton's \cite{Hamilton:WeldingSurvey}.

\begin{figure}[h]
\includegraphics[scale=1]{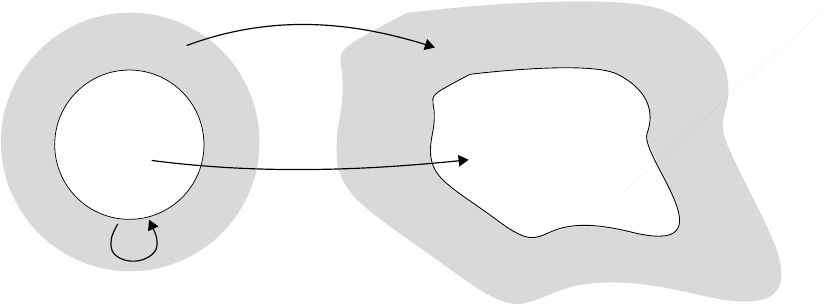}
\centering
\setlength{\unitlength}{\textwidth}
\put(-0.62,0.172){$f$}
\put(-0.62,0.31){$g$}
\put(-0.19,0.26){$\gamma$}
\put(-0.9,0.26){$\Sone$}
\put(-0.772,0.059){$h=g^{-1}\circ f$}
\caption{Definition of conformal welding. The conformal maps $f$ and $g$ map the interior and exterior of the unit circle to the two complementary components of a Jordan curve $\gamma$. The map $h\colon\Sone\to\Sone$ is an orientation-preserving homeomorphism of $\Sone$.}
\label{figure:welding_def}
\end{figure}
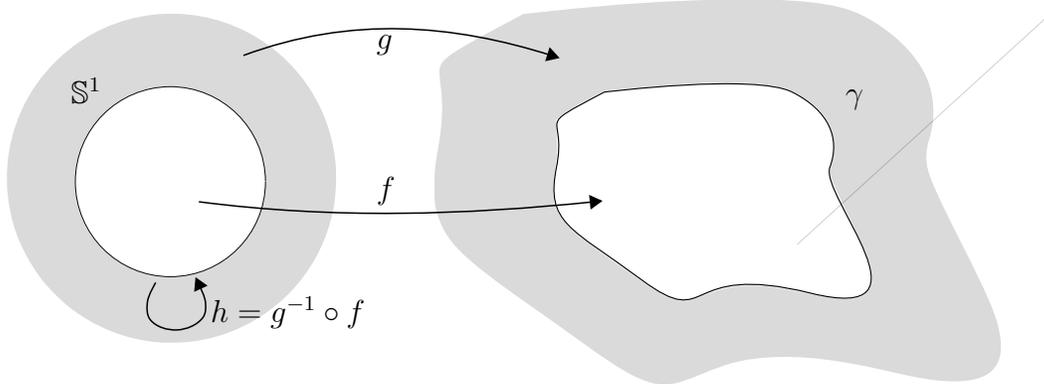

We proceed to explain some implications of Theorem \ref{theorem:non-injectivity}.

\begin{corollary}[\cite{Chris:Borel} Question 11]\label{corollary:uncountable-to-1}
The welding correspondence is uncountable-to-$1$ for log-singular circle homeomorphisms.
\end{corollary}

The curves that arise from log-singular circle homeomorphisms, which will be discussed more extensively later, are called \textit{flexible} and they were first introduced by Bishop \cite{MR1274085}. We say that a Jordan curve $\gamma\subset\C$ is flexible if:
\begin{enumerate}[label=(\roman*)]
	\item Given any other Jordan curve $\tilde{\gamma}$ and any $\eps>0$, there exists a homeomorphism $\phi\colon\C_{\infty}\to\C_{\infty}$, conformal off $\gamma$, so that the Hausdorff distance between $\phi(\gamma)$ and $\tilde{\gamma}$ is less than $\eps$.
	\item Given $z_{1},z_{2}$ in each component of $\C_{\infty}\setminus\gamma$ and $w_{1},w_{2}$ in each component of $\C_\infty\setminus\tilde{\gamma}$, the previous $\phi$ can be taken so that $\phi(z_{j})=w_{j}$ for $j=1,2$.
\end{enumerate}

It was noticed by Bishop \cite{ChrisWeldingAnnals} that an application of a theorem of Balogh and Bonk \cite{BaloghBonk} shows that the welding homeomorphism of a flexible curve is log-singular. In the same paper, Bishop proved that the converse also holds.

The set of homeomorphisms $\phi$ from (i) are in $$\textrm{CH}(\gamma)=\{\phi\colon\C_{\infty}\to\C_{\infty} \textrm{ homeomorphism}\colon\phi\textrm{ conformal off }\gamma\}.$$ If $\textrm{CH}(\gamma)$ consists only of M\"obius transformations, then we say that the curve $\gamma$ is \textit{conformally removable}. See the survey papers by Bishop \cite{Chris:Borel} and Younsi \cite{Younsi:SurveyRemovability} for more information on conformally removable sets. Characterizing conformally removable Jordan curves still remains a difficult open problem, let alone characterizing conformally removable sets. However, it is worth mentioning work of Ntalampekos \cite{Dimitrios:CNED} that constitutes a very important step in this direction. The corresponding problem for bounded analytic functions was solved by Tolsa \cite{Tolsa:Painleve} in his breakthrough work.

It follows from the definition that flexible curves are non-removable. Another proof follows by Theorem \ref{theorem:non-injectivity}, which yields the following stronger result.

\begin{corollary}\label{theorem:CHFlexible}
Let $\gamma\subset\C$ be a flexible curve, then there exists a homeomorphism $\phi\in\textrm{CH}(\gamma)$ so that $\phi(\gamma)$ has positive area.
\end{corollary}

The proof follows by considering the homeomorphism constructed in Figure \ref{figure:non-uniqueness_welding} together with Theorem \ref{theorem:non-injectivity}.

\begin{figure}[h]
\includegraphics[scale=1]{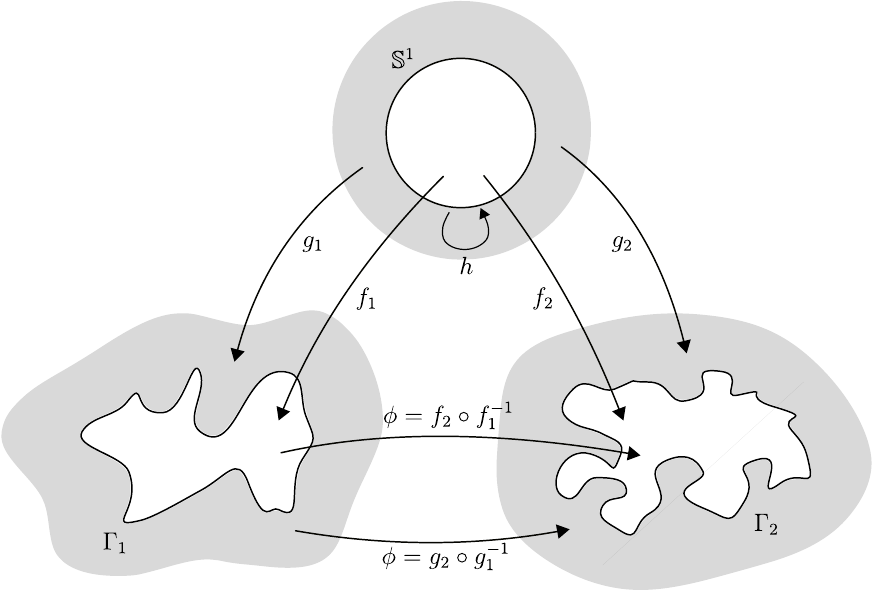}
\centering
\caption{If two non-conformally equivalent curves $\Gamma_{1}, \Gamma_{2}$ yield the same welding $h$, then there exists a homeomorphism $\phi\in\textrm{CH}(\Gamma_{1})$ so that $\phi(\Gamma_{1})=\Gamma_{2}$.}
\label{figure:non-uniqueness_welding}
\end{figure}

In fact, Theorem \ref{theorem:non-injectivity} yields the following even stronger result.

\begin{corollary}\label{theorem:CHFlexible_Hdim}
Let $\gamma\subset\C$ be a flexible curve and $s\in[1,2]$, then there exists a homeomorphism $\phi\in\textrm{CH}(\gamma)$ so that $\phi(\gamma)$ has Hausdorff dimension $s$.
\end{corollary}

Corollary \ref{theorem:CHFlexible} is a particular case of a conjecture that has been open for over 30 years.

\begin{conjecture}[\cite{MR1274085} Question 2]\label{conjecture:positive_area}
Let $\gamma$ be a non-removable curve, then there exists $\phi\in\textrm{CH}(\gamma)$ so that $\phi(\gamma)$ has positive area.
\end{conjecture}

It follows by a result of Pugh and Wu \cite{PughWu} and a result of Beurling \cite{Beurling:exceptionnels} that flexible curves are residual in the space of Jordan curves. Hence Corollary \ref{theorem:CHFlexible} proves Conjecture \ref{conjecture:positive_area} for a \textit{generic} class of Jordan curves.

If Conjecture \ref{conjecture:positive_area} holds, then the following conjecture that has been open for over 50 years also holds.

\begin{conjecture}
The conformal welding correspondence is injective if and only if the corresponding Jordan curve is conformally removable. That is, a Jordan curve is uniquely determined by its welding homeomorphism (modulo M\"obius transformations) if and only if it is conformally removable.
\end{conjecture}

The first instance that I can recall of this conjecture appears in Lehto's \cite{Lehto:Weldings}. This conjecture has been claimed as true by many authors, but either without giving a proof or with an incorrect argument, as it was noticed by Fortier-Bourque. See Younsi's \cite{Younsi:SurveyRemovability}. 

Later, Younsi \cite{Younsi:Welding_Ex} suggested that injectivity of the conformal welding correspondence might not be completely characterized by whether or not $\gamma$ is conformally removable. In the sense that even when $\gamma$ is non-removable and $\phi\in\textrm{CH}(\gamma)$, it could be the case that $\phi(\gamma)$ is a M\"obius image of $\gamma$. More precisely, Younsi \cite[Theorem 1.3]{Younsi:Welding_Ex} proved that there is a flexible curve of zero area and $\phi\in\textrm{CH}(\gamma)$ so that $\phi(\gamma)=\gamma$. However, Corollary \ref{theorem:CHFlexible_Hdim} shows that for flexible curves there is some other $\varphi\in\textrm{CH}(\gamma)$ for which $\varphi(\gamma)$ is not a M\"obius image of $\gamma$.

Notice that if $\gamma$ is assumed to have positive area, then the conformal welding correspondence $\mathcal{W}$ is not injective by the measurable Riemann mapping theorem.

Corollary \ref{theorem:CHFlexible_Hdim} also yields the following.

\begin{corollary}
Let $\gamma$ be a flexible curve, then there is $\phi\in\mathrm{CH}(\gamma)$ so that $\phi(\gamma)$ has zero area.
\end{corollary}

Since by Theorem \ref{theorem:non-injectivity} there are flexible curves of positive area, this also answers a particular case of the following open conjecture.

\begin{conjecture}\label{conjecture:positive-area_to_zero-area}
Let $E$ be a positive area set with no interior, then there exists $\phi\in\mathrm{CH}(E)$ so that $\phi(E)$ has zero area.
\end{conjecture}

This conjecture first appears as a question in Bishop's \cite[Question 3]{MR1274085}. Later, Kaufman and Wu \cite[Theorem 3]{Kaufman-Wu:Removability} proved that there is always $F\subset E$ of positive area and $\phi\in\mathrm{CH}(E)$ so that $\phi(F)$ has zero area, and conjectured that the answer to Bishop's question is negative. That is, that Conjecture \ref{conjecture:positive-area_to_zero-area} is false.

To the best of our knowledge Theorem \ref{theorem:non-injectivity} (and Corollary \ref{corollary:uncountable-to-1}), is the first instance in which a zero area non-removable Jordan curve corresponds to uncountably many non-conformally equivalent Jordan curves via CH homeomorphisms. See Figure \ref{figure:non-uniqueness_welding}. Similarly, our construction in the proof of Theorem \ref{theorem:non-injectivity} also implies.

\begin{corollary}[\cite{Chris:Borel} Question 14]
The set of $CH$-images of a flexible curve is not totally disconnected.
\end{corollary}

Conformal welding has been used in many other areas of mathematics. For example, Sheffield \cite{Sheffield:Welding} and Duplantier, Miller and Sheffield \cite{Sheffield:LQG} used conformal welding to relate Liouville quantum gravity and Schramm-Loewner evolution curves. Ang, Holden and Sun \cite{Ang:SLE_Welding} prove that SLE$_{\kappa}$ measure, for $\kappa\in(0,4)$, arises naturally from the conformal welding of two $\sqrt{\kappa}$-Liouville quantum gravity disks. For some recent work in this direction see Fan and Sung \cite{Jinwoo_Welding}. Sharon and Mumford \cite{Mumford:Welding} used conformal welding in their applications to computer vision. Conformal welding also is the basis of the construction of universal Teichm\"uller space (see for instance Lehto's book \cite{Lehto:Teichmuller}). McMullen \cite{McMullen:QuestionMark} has recently used conformal welding in his applications to complex dynamics.


\subsection{Outline of the paper}\label{section:outline}

We briefly outline now how Theorem \ref{theorem:non-injectivity} is proved by recalling Bishop's original construction \cite[Theorem 25]{ChrisWeldingAnnals}, and by explaining how to sharpen it to produce flexible curves of any possible Hausdorff dimension (and even positive area). Our construction uses quasiconformal mappings. Observe that we can also define \textit{quasiconformal weldings} by taking the maps in the definition of conformal welding to be quasiconformal instead of conformal. However, by the measurable Riemann mapping theorem we obtain:

\begin{proposition}\label{proposition:qcw}
Let $\gamma$ be a Jordan curve with quasiconformal welding $h$, then there exists a quasiconformal map $H\colon\C\to\C$ so that $H(\gamma)$ has conformal welding $h$. That is, quasiconformal weldings are the same as conformal weldings. 
\end{proposition}


In \textbf{Section \ref{section:log_capacity}} we introduce all the preliminaries, which include logarithmic capacity, Hausdorff measures and dimension, flexible curves, extremal length, quasiconformal mappings and harmonic measure.

 Let $h\colon\Sone\to\Sone$ be a log-singular circle homeomorphism, and suppose that we have two quasiconformal mappings $f_{1}\colon\D\to\Omega_{1}$ and $g_{1}\colon\C\setminus\overline{\D}\to\Omega_{1}^{*}$, where $\Omega_{1}, \Omega_{1}^{*}$ are Jordan domains with disjoint closures. The goal is to obtain new quasiconformal mappings $f_{2}\colon\D\to\Omega_{2}$ and $g_{2}\colon\C\setminus\overline{\D}\to\Omega_{2}^{*}$, with $\Omega_{2}$ and $\Omega_{2}^{*}$ Jordan domains with disjoint closures, so that loosely speaking
$$ ||(f_{2})_{|\Sone}-(g_{2}\circ h)_{|\Sone}||_{\infty} \leq c ||(f_{1})_{|\Sone}-(g_{1}\circ h)_{|\Sone}||_{\infty},$$
where $c<1$. The quasiconformal maps $f_{2}, g_{2}$ will be constructed, loosely speaking, as \textit{quasiconformal extensions} of $f_{1}, g_{1}$. This is illustrated in Figure \ref{figure:strategy_paper}. As long as we can iterate this process while keeping the quasiconformality constant of our mappings uniformly bounded, it follows by Proposition \ref{proposition:qcw} that $h$ is a conformal welding. The aforementioned maps will not exist in general, but Bishop \cite{Chris:boundary_interpolation, ChrisWeldingAnnals} astutely observed that this can be done on a set of zero logarithmic capacity. Since $h$ is log-singular such a sequence will exist and $h$ will be a conformal welding.

\begin{figure}[h]
\includegraphics[scale=1]{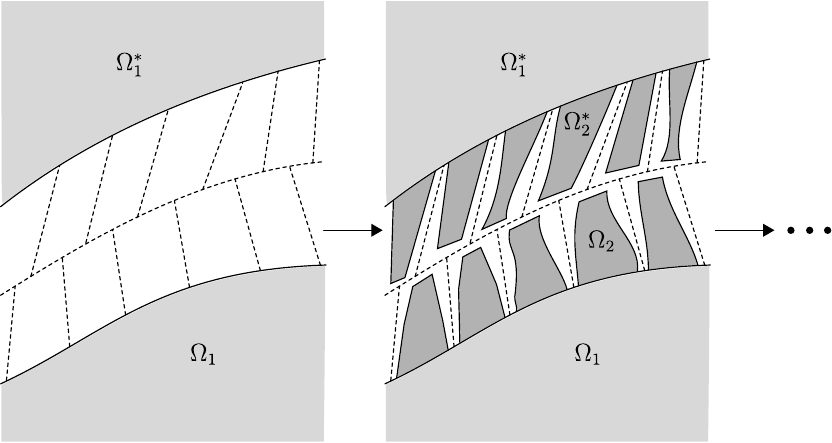}
\centering
\caption{The curve with welding $h$ is constructed via an iterative procedure, where at each step we quasiconformally extend maps $f_{1}\colon\D\to\Omega_{1}$, $g_{1}\colon\C\setminus\overline{D}\to\Omega_{1}^{*}$ in such a way that the infinity norm between $f_{1}(\xi)$ and $g_{1}(h(\xi))$ is decreased by a fixed ratio.}
\label{figure:strategy_paper}
\end{figure}

Given the maps $f_{1}, g_{1}$ we construct the maps $f_{2},g_{2}$ in the following way. First, given a closed set $E\subset\Sone$, consisting of a finite union of closed arcs and with small logarithmic capacity, we construct a conformal map $\phi\colon\D\to W$, where $W$ is roughly speaking a large disk $\D(0,R)$ of radius $R$ (where $R$ is large if the logarithmic capacity of $E$ is small) with a finite number of radial slits removed, which all intersect the unit circle $\Sone$. This map is constructed in \textbf{Section \ref{section:construction1}} and represented in Figure \ref{figure:phi_n}. Notice that $W\setminus\D$ consists of a finite union of quadrilaterals $W_{k}$ (we specify later the quad-vertices and sides).

Second, we construct a quasiconformal extension of $f_{1}$ (and in an analogous way of $g_{2}$). We emphasize that controlling the quasiconformality constant of a quasiconformal extension requires having plenty of geometric control over where it is constructed. Observe that we can embed the quadrilaterals $W_{k}$ (obtained when constructing $\phi$) into the annulus $\C\setminus(\Omega_{1}\cup\Omega_{1}^{*})$, by mapping the $W_{k}$ to the dark colored quadrilaterals in Figure \ref{figure:strategy_paper}. These embeddings can be taken to be conformal. However, in general they will not extend our mapping $f_{1}$. To address this, we introduce in \textbf{Section \ref{section:shapes}} the notion of \textit{admissible shapes}, which is the main novel technique of this paper. In short, we quantify when we will be able to define an extension that interpolates between our mapping $f_{1}$ and these conformal embeddings, and that has quasiconformality constant as close to the one of $f_{1}$ as necessary. See Figure \ref{figure:definition_shapes_intro} for an example.

In \textbf{Section \ref{section:construction2}} we use admissible shapes to construct $\psi$, the aforementioned quasiconformal extension of $f_{1}$ (and similarly we can define it for $g_{1}$).

In \textbf{Section \ref{section:logsingular}} we define $f_{2}$. Roughly speaking $f_{2}=\psi\circ\phi$. We complete the proof of Bishop's result \cite[Theorem 25]{ChrisWeldingAnnals}.

As we will see in Section \ref{section:shapes}, one of the main advantage of admissible shapes is that they offer plenty of \textit{flexibility} when choosing the embeddings into the annulus $\C\setminus(\Omega_{1}\cup\Omega_{1}^{*})$. In \textbf{Section \ref{section:positive_area}} we use this to prove that there is a flexible curve of positive area (thus Hausdorff dimension $2$) with welding $h$.

Recall that when using Proposition \ref{proposition:qcw} to recover a Jordan curve with conformal welding $h$, the corresponding quasiconformal map could change Hausdorff dimension. This topic has a very rich history, see for instance Astala's \cite{Astala-distortion}, Gehring's \cite{MR324028} and Smirnov's \cite{Smirnov_quasicircles}. Since quasiconformal maps preserve null-sets, this is not a problem to prove that there are flexible curves of positive area. However, it poses a problem when proving that there are flexible curves of Hausdorff dimension $s\in[1,2)$.

In \textbf{Section \ref{section:dim_1}} we use shapes, together with some bounds on how a planar quasiconformal map with small quasiconformality constant differs from a linear transformation, to prove Theorem \ref{theorem:non-injectivity} for $s=1$.

Finally, in \textbf{Section \ref{section:dim_s}} we use an explicit construction, together with the results from Section \ref{section:shapes} to prove Theorem \ref{theorem:non-injectivity} for $s\in(1,2)$.

\begin{figure}[h]
\includegraphics[scale=1]{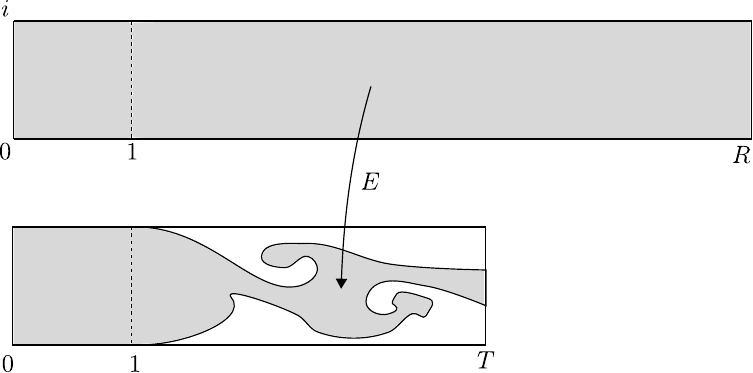}
\centering
\caption{Shapes encode conformal embeddings of a quadrilateral into another quadrilateral. The map $E$ from the rectangle $[0,R]\times[0,1]$ into the rectangle $[0,T]\times[0,1]$ is conformal.}
\label{figure:definition_shapes_intro}
\end{figure}

 \subsection*{Acknowledgements} 

Part of the work in this paper was done during the author's visit to the Hausdorff Research Institute of Mathematics. I would specially like to thank Chris Bishop for many conversations, continuous feedback and corrections. I also thank Kari Astala, Stathis Chrontsios, Hrant Hakobyan, Randall Kayser, Willie Lim, Misha Lyubich, Dimitrios Ntalampekos, Eero Saksman, Raanan Schul, Jinwoo Sung, Malik Younsi and Catherine Wolfram for conversations regarding the results in this paper and their comments.


\subsection*{Notation} 

We write $a\lesssim b$ to denote that there is some $C<\infty$ so that $a\leq C b$. We write $a\simeq b$ if $a\lesssim b$ and $b\lesssim a$. When introducing a new definition in the middle of result statement or a proof we will use $\coloneqq$.

As usual, $\D=\{|z|<1\}$ denotes the unit disk and $\Sone=\{|z|=1\}$ the unit circle. Similarly, $\D(a,R)=\{|z-a|<R\}$, $\D^{*}=\{|z|>1\}$ and $\C_{\infty}=\C\cup\{\infty\}$ is the Riemann sphere. $[a,b]$ denotes the line segment joining $a,b$, but when it is clear from the context it will denote shortest circular arc in $\Sone$ joining $a,b$. We will denote by $|[a,b]|$ the length of such interval or arc.

Unless we specify the opposite (by saying either \textit{into} or \textit{onto}) the maps we consider will be homeomorphisms.


\section{Preliminaries}\label{section:log_capacity}

In this section we summarize known results regarding (logarithmic) capacity that we will use in the following sections. These results are all contained in the books of Carleson \cite{CarlesonBook}, Garnett and Marshall \cite{HarmonicMeasure}, and Pommerenke \cite{PommerenkeBook}. We also recall the definition of Hausdorff measures and dimension, some previously known results about flexible curves, extremal length and quasiconformal mappings, and harmonic measure.


\subsection{Logarithmic capacity}

Let $\mu$ be a finite compactly supported signed (Borel) measure. The \textit{logarithmic potential} of $\mu$ is the functional $$ U_{\mu}(z)=\int \log\frac{1}{|\xi-z|}d\mu(\xi).$$

By Fubini's theorem, $U_{\mu}(z)$ is finite a.e. with respect to Lebesgue measure. If $$ \iint\left| \log\frac{1}{|z-\xi|}\right| d|\mu|(\xi)d|\mu|(z)<\infty,$$ we say that $\mu$ has \textit{finite energy} and define the \textit{energy integral} $I(\mu)$ by
$$ I(\mu)=\iint \log\frac{1}{|z-\xi|} d\mu(\xi)d\mu(z)=\int U_{\mu}(z)d\mu(z).$$

Let $K$ be compact and denote by $P(K)$ the set of all Borel probability measures on $K$. Define the \textit{Robin's constant} of $K$ by 
$$ \gamma(K)=\inf\{I(\mu)\colon \mu\in P(K)\},$$
and the (logarithmic) \textit{capacity} of $K$ by 
$$ \capacity(K)=e^{-\gamma(K)}.$$
It follows that $\gamma(K)<\infty$ if and only if $\capacity(K)>0$, which holds as long as there is a probability $\mu\in P(K)$ with finite energy $I(\mu)<\infty$. If $K_{1}\subset K_{2}$, then any probability on $K_{1}$ is also a probability on $K_{2}$. Hence $\gamma(K_{1})\geq\gamma(K_{2})$ and $\capacity(K_{1})\leq\capacity(K_{2})$, i.e. the capacity set function is monotone with respect to set inclusion.

The \textit{capacity} of any Borel set $E$ is defined as 
$$\capacity(E)=\sup\{\capacity(K)\colon K \textrm{ compact}, K\subset E\}.$$

Capacity can be a difficult object to compute, but the capacities of some simple sets are know (see \cite[Chapter 3]{HarmonicMeasure} or \cite[Chapter 9]{PommerenkeBook}). For example, for $a,b\in\C$, $\capacity[a,b]=|a-b|/4$, and the capacity of a disk is its radius, i.e. $\capacity(\D(a,r))=r$.

We now summarize some other known properties of logarithmic capacity that we will need later.

\begin{proposition}\label{proposition:capacity_properties} Logarithmic capacity satisfies the following properties.
\begin{enumerate}[label=(\roman*)]
	\item If $\varphi$ is $L$-Lipschitz, then $\capacity(\varphi(E))\leq L\capacity(E)$.
	\item If $h(\xi)=c\xi+c_{0}+c_{1}\xi^{-1}+\cdots$ maps $\C_{\infty}\setminus\overline{\D}$ conformally onto $G$ then $\capacity(\partial G)=|c|$.
	\item Borel sets are capacitable, that is, given $E$ Borel and $\eps>0$, there is an open set $V$ with $E\subset V$ and $\capacity(V)<\capacity(E)+\eps$.
	\item If $E_{n}$ are Borel sets so that $E=\cup E_{n}$ has diameter less or equal than one, then $1/\gamma(E)\leq\sum1/\gamma(E_{n})$.
\end{enumerate}
\end{proposition}

The proof of (i), (ii) and (iii) can be found in \cite[Chapter 9]{PommerenkeBook},  (iv) follows from Lemma 4 in \cite[Section 3]{CarlesonBook}.


\subsection{Hausdorff measures and dimension}

We follow \cite[Chapter II]{CarlesonBook}.

A measure function is $h\colon\R_{+}\to\R_{+}$ continuous and increasing so that $h(0)=0$. $h$ is also called in the literature a \textit{gauge function}. Let $E$ be bounded and consider all possible coverings of $E$ with a countable number of disks $D_{j}$ with radii $r_{j}$, i.e. $E\subset\cup_{j}D_{j}$. We define, for $\delta\in(0,\infty]$,
\begin{equation}
	H^{\delta}_{h}(E)=\inf\left\{\sum_{j} h(r_{j})\colon E\subset\cup_{j} D_{j}, r_{j}\leq\delta\right\}.
\end{equation}
The limit,
\begin{equation}
	H_{h}(E)=\lim_{\delta\to0} H_{h}^{\delta}(E)
\end{equation}
exists (even though it can potentially be infinite) and is called the \textit{Hausdorff measure} associated to $h$. $H^{\infty}_{h}$ is called the \textit{Hausdorff content}, which is not a measure, but $H^{\infty}_{h}$ and $H_{h}$ vanish on the same sets.

If $h(r)=r^{s}$ we denote the measure $H_{h}$ by $H_{s}$ and we define
$$ \mathcal{H}\textrm{dim}(E)=\inf\{s>0\colon H_{s}(E)=0\}$$ the Hausdorff dimension of $E$. When proving that a certain set has some Hausdorff dimension we often find an upper bound by considering explicit coverings. To find a lower bound, one needs to construct a Frostman measure, which is a measure as in Theorem \ref{theorem:frostman}.

\begin{theorem}[Frostman's Lemma]\label{theorem:frostman}
Let $\mu$ be a measure so that $\mu(D)\leq h(r)$ for every disk $D$ or radius $r$, then $\mu(E)\leq H_{h}(E)$. Conversely, there is a constant $c$ such that for every compact set $F$, there is a $\mu$ satisfying $\mu(D_{r})\leq h(r)$ so that $\mu(F)\geq c H^{\infty}_{h}(F)$.
\end{theorem}

A very elegant proof of Theorem \ref{theorem:frostman} can be found in Carleson's book \cite{CarlesonBook}, even though the first proof appeared in Frostman's PhD thesis. Observe that Theorem \ref{theorem:frostman} only provides useful information if $H^{\infty}_{h}(E)>0$, i.e. when the measure $\mu$ is non-zero. When applied to $h(r)=r^{s}$ and $\mu$ as in Theorem \ref{theorem:frostman}, it follows that $\mathcal{H}\textrm{dim}(E)\geq s$.

\subsection{Flexible curves and log-singular homeomorphisms}

In \cite{MR1274085} Bishop introduced and proved the existence of flexible curves. We say that a Jordan curve $\gamma\subset\C$ is \textit{flexible} if the following conditions are satisfied:
\begin{enumerate}[label=(\alph*)]
	\item Given any other Jordan curve $\tilde{\gamma}$ and any $\eps>0$, there exists a homeomorphism $\phi\colon\C_{\infty}\to\C_{\infty}$, conformal off $\gamma$, so that the Hausdorff distance between $\phi(\gamma)$ and $\tilde{\gamma}$ is less than $\epsilon$.
	\item Given $z_{1},z_{2}$ in each component of $\C_{\infty}\setminus\gamma$ and $w_{1},w_{2}$ in each component of $\C_\infty\setminus\tilde{\gamma}$, the previous $\phi$ can be taken so that $\phi(z_{j})=w_{j}$ for $j=1,2$.
\end{enumerate}

Flexible curves are conformally non-removable in a very strong sense (as we will also see in the proof of Theorem \ref{theorem:CHFlexible_Hdim}). Later in \cite{ChrisWeldingAnnals} Bishop proved that $h\colon\Sone\to\Sone$ is the conformal welding of a flexible curve if and only if $h$ is log-singular. The specific result that he proved is Theorem \ref{theorem:log-singular_to_flexible}, which will be stated later in Section \ref{section:logsingular}.

We have the following characterization of log-singular homeomorphisms.

\begin{lemma}[\cite{ChrisWeldingAnnals} Lemma 11]\label{lemma:equivalence_logsingular}
Suppose $h\colon\Sone\to\Sone$ is an orientation-preserving circle homeomorphism. Then the following are equivalent.
\begin{enumerate}[label=(\roman*)]
	\item For any $\eps>0$ there is a finite union of closed intervals $E\subset\Sone$ such that $\capacity(E)+\capacity(h(\Sone\setminus E))<\eps$. 
	\item For any $n\in\N$ there is a compact set $E_{n}\subset\Sone$ so that $\capacity(E_{n})\leq 1/n$ and $\capacity(h(\Sone\setminus E_{n}))\leq 1/n$.
	\item There is a Borel set $E$ so that both $E$ and $h(\Sone\setminus E)$ have zero logarithmic capacity, i.e. $h$ is log-singular.
\end{enumerate}
\end{lemma}

Pugh and Wu \cite[Theorem 3]{PughWu} constructed some special Jordan curves that arise as funnel sections of some ODEs. The Jordan curves $\gamma$ they construct satisfy the following property: any rectifiable curve $\sigma$ that has one endpoint in each one of the complementary components of $\gamma$, crosses $\gamma$ in more than one point (they call this condition being \textit{non-smoothly pierceable}). In the same paper they also proved that such curves are \textit{generic} \cite[Theorem 15]{PughWu}, i.e. they are residual set in the space of Jordan curves with the topology they define. These curves are also flexible; a result of Beurling \cite{Beurling:exceptionnels} says that the hyperbolic geodesics ending at a point of the boundary of a simply connected domain all have finite length, for except the image of a set of zero logarithmic capacity on $\Sone$.

Flexible curves can also exhibit very pathological behavior. For instance, motivated by the results from Pugh and Wu, in \cite{Burkart:JordanCurve} Burkart constructs a flexible curve that cannot be crossed by a rectifiable arc on a set of zero length (in particular such curve has positive area).

Log-singular circle homeomorphisms have also found applications. For example, in \cite{Alex:Welding} I show that every circle homeomorphism can be written as the composition of two log-singular circle homeomorphisms.


\subsection{Extremal length and modulus of path families}

Let $\Gamma$ be a family of curves in $\C$, where each $\gamma\in\Gamma$ is a countable union of open arcs, closed arcs or closed curves, and every closed arc is rectifiable. A non-negative Borel function $\rho\colon\C\to[0,\infty)$ is admissible for $\Gamma$ if 
$$ l(\Gamma)=l_{\rho}(\Gamma)=\inf_{\gamma\in\Gamma}\int_{\gamma}\rho ds\geq1.$$
We define the modulus $M(\Gamma)$ of this path family as 
$$ M(\Gamma)=\inf_{\rho} \iint{\rho^{2}}dm$$
and the extremal length $\lambda(\Gamma)=1/M(\Gamma)$.

By using Cauchy-Schwarz we can see that given the rectangle $[0,R]\times[0,1]$ and the path family that joins the vertical sides, then $M(\Gamma)=1/R$. Also, we can use this to prove that if $A_{r,R}(z_{0})=\{r<|z-z_{0}|<R\}$ and $\Gamma_{A}$ is the path family that separates the two components in the annulus, then 
$$ M(\Gamma_{A})=\frac{1}{2\pi}\log(R/r).$$

By a Jordan quadrilateral $Q=Q(v_{1},v_{2},v_{3},v_{4})\subset\C$, we mean a bounded Jordan domain with four marked points, which we call \textit{quad-vertices} (and we suppose that they are oriented counter-clockwise). To abbreviate, we will refer to them just as quadrilaterals. We denote the Jordan arcs joining these points by $a_{1}, b_{1}, a_{2}$ and $b_{2}$ as it is indicated in Figure \ref{figure:modulusQuad}, which we call the \textit{sides} of the quadrilateral. We refer to each pair of non-intersecting sides as opposite sides, and we call $a_{1}, a_{2}$ the $a$-sides of $Q$ and $b_{1},b_{2}$ the $b$-sides of $Q$.

\begin{figure}[h]
	\includegraphics[scale=1]{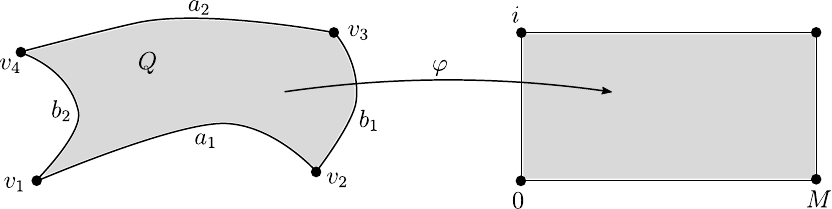}
	\centering
	\caption{Representation of the quad-vertices, $a$-sides and $b$-sides of a quadrilateral Q, together with its representation via conformal mapping $\varphi$ to a rectangle $[0,M]\times[0,1]$, which maps quad-vertices to vertices. The value $M$ is a conformal invariant.}
	\label{figure:modulusQuad}
\end{figure}

A quadrilateral $Q$ can be conformally mapped to a rectangle so that the quad-vertices of the quadrilateral are mapped to the vertices of the rectangle (see 6.2.3 in \cite{AhlforsCA} or Chapter I, Section 2.4 in \cite{LehtoVirtanen}). The ratio between the length of the $a$-side and the $b$-side of this rectangle is a conformal invariant. The quantity $M=M(Q)=M(Q(v_{1},v_{2},v_{3},v_{4}))$, as in Figure \ref{figure:modulusQuad}, is called the \textit{modulus} of $Q$.

The modulus is a conformal invariant. If we define the path family $\Gamma$ to consist of all those locally rectifiable paths joining opposite $a$-sides of a quadrilateral $Q$, then the modulus of this path family agrees with the modulus of the quadrilateral, i.e. $M(\Gamma)=M(Q)$.

We summarize some of the main properties.

\begin{proposition}[Properties of modulus and extremal length]\label{proposition:modulus_properties}
Let $\Gamma_{k}, \Gamma, \Gamma'$ be path families.
\begin{enumerate}[label=(\roman*)]
	\item (Monotonicity) If every $\gamma'\in\Gamma'$ contains some $\gamma\in\Gamma$, then $\la(\Gamma')\geq\la(\Gamma)$ and $M(\Gamma')\leq M(\Gamma)$.
	\item (Serial Rule) If $\Gamma\supset\Gamma_{1}+\Gamma_{2}$, then $\la(\Gamma)\geq\la(\Gamma_{1})+\la(\Gamma_{2})$.
	\item (Parallel Rule) If $\Gamma_{1},\Gamma_{2}$ lie in disjoint Borel sets and $\Gamma=\Gamma_{1}\cup\Gamma_{2}$, then $M(\Gamma)\geq M(\Gamma_{1})+M(\Gamma_{2})$.
	\item (Monotonicity+Countable subbaditivity) If $\Gamma\subset\Gamma'$ or if for every $\gamma\in\Gamma$ there is $\gamma'\in\Gamma'$ with $\gamma'\subset\gamma$, then $M(\Gamma)\leq M(\Gamma')$. Furthermore,
	$$ M\left(\bigcup_{k}\Gamma_{k}\right)\leq \sum_{k}M(\Gamma_{k}).$$
	\item If $\cup_{k}\Gamma_{k}\subset\Gamma$ and the supports of $\Gamma_{k}$ live in disjoint Borel sets, then $\sum_{k}M(\Gamma_{k})\leq M(\Gamma)$.
	\item If for each $\gamma\in\Gamma$ there is $\gamma_{k}\in\Gamma_{k}$, then $\la(\Gamma)\geq\sum_{k}\la(\Gamma_{k})$.
\end{enumerate}
\end{proposition}

We will also need:

\begin{lemma}[\cite{AhlforsQC}]\label{lemma:ahlfors_modulus}
$m=m_{1}+m_{2}$ if and only if the dividing line is $\{x=m_{1}\}$.
\end{lemma}

The following result due to Balogh and Bonk (which can be proved using extremal length), will be of great use in this paper.

\begin{lemma}[\cite{BaloghBonk}, \cite{ChrisWeldingAnnals} Lemma 17]\label{lemma:capacity_far}
Suppose $f\colon\D\to\Omega$ is conformal and for $R\geq1$ we set
$$ E=\{x\in\Sone\colon |f(x)|\geq R\cdot\textrm{d}(f(0),\partial\Omega)\}.$$
Then $\capacity(E)\leq C R^{-1/2}$, with $C$ independent of $\Omega$.
\end{lemma}


\subsection{Quasiconformal mappings}

Quasiconformal mappings were introduced by H. Grötzsch in 1928. Since the modulus is a conformal invariant, there is no conformal map that maps a square to a rectangle (which is not a square) mapping vertices to vertices. There are several (equivalent) definitions of planar quasiconformal mappings (see \cite{AhlforsQC, LehtoVirtanen} for more details). We say that an orientation preserving homeomorphism $\phi\colon\Omega\subset\C\to\Omega^{'}\subset\C$ is $K$-quasiconformal if for every quadrilateral $Q\subset\Omega$, we have $$M(\phi(Q))\leq K M(Q).$$ We call $K\geq1$ the \textit{quasiconformality constant} of $\phi$ or its \textit{maximal dilatation}.

The previous definition, which is often referred as the geometric definition, is equivalent to; $\phi$ is an orientation preserving homeomorphism with locally integrable distributional derivatives $\phi_{z}, \phi_{\overline{z}}$ which satisfy
$ |\phi_{\overline{z}}|\leq k|\phi_{z}| \textrm{ a.e.}$, where $k=(K-1)/(K+1)$ (we call this the analytic definition). This inequality also guarantees that $\phi\in W^{1,2}_{\textrm{loc}}$. The measurable function $\mu_{\phi}=\phi_{\overline{z}}/\phi_{z}$ is called the \textit{dilatation} of $\phi$. We will also refer to $\phi_{\overline{z}}=\mu_{\phi} \phi_{z}$ as the \textit{Beltrami equation}.

For more definitions, and the corresponding extensions to the metric and higher dimensional setting, the reader is encouraged to read Heinonen and Koskela's paper \cite{Heinonen:DefinitionsQC} and the more recent work of Ntalampekos \cite{Dimitrios:metric_defQC}.

One of the cornerstones of the theory of planar quasiconformal mappings is the so-called measurable Riemann mapping theorem. 

\begin{theorem}[Ahlfors-Bers, \cite{AhlforsQC} Chapter V]\label{theorem:MRMT}
Let $\mu\colon\C\to\D$ be a measurable function so that $||\mu||_{\infty}=k<1$. Then there exists a $(1+k)/(1-k)$-quasiconformal map $\phi\colon\C\to\C$ with dilatation $\mu$. That is, $\phi$ satisfies
$$ \phi_{\overline{z}}=\mu \phi_{z}.$$
Moreover, if we normalize $\phi$ so that it fixes $0,1,\infty$ the mapping is unique.
\end{theorem}

See Astala's et al \cite{AstalaBook} for more extensions of this result for different classes of measurable functions $\mu$.

Ahlfors' book also proves the following formula regarding how the solutions of two similar Beltrami equations differ from each other.

\begin{theorem}[\cite{AhlforsQC} Theorem 5, Chapter V]\label{theorem:Ahlfors_formula}
Suppose $$\mu(t+s,z)=\mu(t,z)+s \nu(t,z)+s\eps(s,t,z),$$ where $\nu,\mu,\eps\in L^{\infty}$, $||\mu||_{\infty}<1$ and $||\eps(s,t)||_{\infty}\to0$ as $s\to0$. Then if $f^{\mu(\cdot)}$, denotes the unique solution of the corresponding Beltrami equation that fixes $0,1,\infty$, we have
$$ f^{\mu(t+s)}(\zeta)=f^{\mu(t)}+s \dot{f}(\zeta,t)+o(s)$$ uniformly on compact sets, where
$$ \dot{f}(\zeta,t)=-\frac{1}{\pi}\iint\nu(t,z) R\left(f^{\mu(t)}(z), f^{\mu(t)}(\zeta)\right) \left(f_{z}^{\mu(t)}(z)\right)^{2}dm(z)$$
and $$ R(z,\zeta)=\frac{1}{z-\zeta}-\frac{\zeta}{z-1}+\frac{\zeta-1}{z}=\frac{\zeta(\zeta-1)}{z(z-1)(z-\zeta)}.$$
Moreover, if $\nu(t,z)$ depends continuously on $t$ (in the $L^{\infty}$ sense) then $(\partial/\partial t) f(z,t)$ is a continuous function of $t$.
\end{theorem}


\subsection{Harmonic Measure}

In a simply connected domain $\Omega$, the harmonic measure of a subset $I\subset\Omega$ with respect to a point $z\in\Omega$ is defined as the probability measure on $\partial\Omega$ that a Brownian motion starting from $z$ first hits $\partial\Omega$ in $I$. We denote it by $\omega(z, I, \Omega)$. Since planar Brownian motion is preserved under conformal transformations \cite[Section 7.9]{Chris:fractal_book}, if $f\colon\D\to\Omega$ is a conformal map with $f(0)=z$ it follows that $\omega(z,I,\Omega)$ is the normalized arc-length of $f^{-1}(I)$. If $K\subset\C$ is a compact set with positive logarithmic capacity, $\omega(\infty, \cdot, \C\setminus K)$ corresponds to the probability measure that minimizes the energy integral. See Garnett and Marshall's \cite[Chapter III]{HarmonicMeasure} for more information. 

In this paper we will use the following harmonic measure estimate.

\begin{lemma}[\cite{HarmonicMeasure} Lemma 5.1, p.143]\label{lemma:HM_rectangle}
Let $R_{L}=\{z\colon|\mathrm{Re}z|< L, |\mathrm{Im}z|<1\}$. If $E_{L}=\{z\in\partial R_{L}\colon |\mathrm{Re}(z)|=L\}$ is the union of the vertical edges of $R_{L}$, then 
$$ e^{-\pi L/2}\leq\omega(0,E_{L},R_{L})\leq\frac{8}{\pi}e^{-\pi L/2}.$$
\end{lemma}


\section{Construction of $\phi$}\label{section:construction1}

In this section we construct in Proposition \ref{proposition:phi} the conformal map $\phi$ that we mentioned in the introduction, which is the first step towards proving Bishop's result \cite[Theorem 25]{ChrisWeldingAnnals}. $\phi$ can be constructed in several ways, for instance one can take the map defined in \cite[Proposition 9.15]{PommerenkeBook} or in \cite[Lemma 26, Lemma 27]{ChrisWeldingAnnals}. For all intents and purposes Pommerenke's method would suffice, but we follow Bishop's approach, which we also expand and sharpen.

The mapping $\phi$ from Proposition \ref{proposition:phi} will be of the form $\exp(U+i\tilde{U})$, where $U(z)=G(z)+G(1/\bar{z})$ and $G$ is the logarithmic potential of a probability measure supported on some set $E\subset\Sone$ (which will correspond to union of closed intervals as in Lemma \ref{lemma:equivalence_logsingular}) that has small capacity. Here we rely on that, in our case, this set $E$ can be taken so that $h(\Sone\setminus E)$ also has small capacity.

\begin{lemma}\label{lemma:construct_E}
	Let $h\colon\Sone\to\Sone$ be a log-singular circle homeomorphism. For $N\in\N$ fixed, there is $A=A(N)>N$ and a set compact set $E=E(A)$ with Robin constant $A$ so that,
	\begin{enumerate}[label=(\roman*)]
		\item The compact set $E$ consists of a finite union of disjoint non-trivial closed arcs so that $ \capacity(E)+\capacity(h(\Sone\setminus E))<e^{-\gamma},$ for $\gamma=\gamma(A)$.
		\item Every arc of length $1/N$ intersects $E$ in a set of positive length, hence in a set of positive logarithmic capacity.
	\end{enumerate}
\end{lemma}
\begin{proof}
	By Lemma \ref{lemma:equivalence_logsingular} given $\eps>0$ there exists a compact set $E=E(\eps)$ which consists of a finite union of disjoint non-trivial closed arcs so that $$ \capacity(E)+\capacity(h(\Sone\setminus E))<\eps.$$ Given $N\in\N$, then an arc of length $1/N$ is mapped under $h$ to an arc of length $\eta=\eta(h,N)>0$. Since the length of an arc is comparable to its logarithmic capacity, if an arc of length $1/N$ does not intersect $E$, then 
$$ \capacity(h(\Sone\setminus E))\gtrsim\eta.$$

Hence fixed $N\in\N$, there is $\eps=\eps(N)>0$ small enough so that any arc of length $1/N$ intersects $E$ in a set of positive length. Equivalently, this holds if $A>N$ the Robin constant of $E=E(\eps)$ is large enough. \end{proof}

The main goal of this section is to prove:

\begin{proposition}[Construction of $\phi$]\label{proposition:phi}
Let $h\colon\Sone\to\Sone$ be a log-singular circle homeomorphism. Suppose $N$ is a large natural number and consider $E=E(N)$ as in Lemma \ref{lemma:construct_E}, which has large Robin constant $A>N$. Let also $M\in\N$ be large and $\delta>0$ small, depending on $N$ and $A$. Then there is $\phi\colon\D\to W$ conformal that satisfies:
\begin{enumerate}[label=(\roman*)]
	\item $W$ is obtained from a Jordan domain $\tilde{W}$ from which we remove a finite number of radial slits. In fact, $\tilde{W}$ is close to a disk, in the sense that there is a universal $C>0$ so that
	$$ \D\left(0, \exp\left(\frac{A}{N}-C\frac{\log N}{N}\right)\right)\subset \tilde{W}\subset  \D\left(0, \exp\left(\frac{A}{N}+C\frac{\log N}{N}\right)\right).$$
	\item $\phi(E)=\partial\tilde{W}$, i.e. $\phi(E)$ is a Jordan curve, and each component of $\Sone\setminus E$ is mapped to a radial slit. 
	\item We can choose $M$ points $\{x_{k}\}_{1}^{M}\subset\partial W$ so that $|x_{k}-\exp(2\pi i k/M)|\leq\delta$ and each $x_{k}$ corresponds to one end of a radial slit.
\end{enumerate}
The radial slits that end at the points $\{x_{k}\}$ yield quadrilaterals $W_{k}\subset\tilde{W}\setminus\D$, where we consider the path family separating the slit corresponding to $x_{k}$ with the slit corresponding to $x_{k+1}$. The $W_{k}$ have modulus
	$$ R_{k}\coloneqq M(W_{k})\simeq\frac{\pi N |x_{k+1}-x_{k}|}{A}\simeq\frac{\pi}{A}\frac{N}{M},$$
where we can assume $M\simeq N$.

Moreover, $\phi$ can be chosen to be as close to the identity as needed in compact sets of $\D$, provided $N$ is large enough.
\end{proposition}

The bound on the modulus $R_{k}$ of $W_{k}$ from Proposition \ref{proposition:phi} does not appear in \cite{ChrisWeldingAnnals}, even though it follows from the arguments there.

We prove now Proposition \ref{proposition:phi}, which consists of several lemmas. 

Write $\Sone=\cup_{k=1}^{N} I_{k}$, where the $I_{k}$ are arcs of the same length $1/N$ that have disjoint interiors. Define $E_{k}=E\cap I_{k}$, for $k=1,\ldots, N$ and set $\mu_{k}$ to be the equilibrium measure of mass $1/N$ on each $E_{k}$ (which is well-defined by our choice in Lemma \ref{lemma:construct_E}). Let $\mu=\sum\mu_{k}$ and define the potential
$$ G(z)=\int\log\frac{1}{|x-z|}d\mu(x)=\sum_{k}G_{k}(x)=\sum_{k}\int\log\frac{1}{|x-z|}d\mu_{k}(z).$$

Observe that $N G_{k}(z)-A$ is the Green's function for $\C_{\infty}\setminus E_{k}$ with pole at $\infty$. Thus $G_{k}(x)=A/N$ on $E_{k}$.

\begin{lemma}[\cite{ChrisWeldingAnnals} Lemma 26]\label{lemma:construction_1}
The map $G$ is continuous on $\C$, harmonic on $\C\setminus E$ and
\begin{enumerate}[label=(\roman*)]
	\item $G(z)\to\log|z|^{-1}$ for $|z|>1$ and $G(z)\to0$ for $|z|<1$ as $N\to\infty$.
	\item For $\delta>0$ and any arc $J\subset\Sone$, $|\{x\in J\colon |G(x)|>\delta\}|/|J|\to0$ as $N\to\infty$.
\end{enumerate}
Moreover, the rate of convergence in (ii) is independent of $A$, provided that $A$ is large enough (where $A$ is the Robin constant of $E$). We also have that for $x\in E$,
\begin{equation}\label{estimate_1} \frac{A}{N}-O\left(\frac{\log N}{N}\right) \leq G(x)\leq\frac{A}{N}+O\left(\frac{\log N}{N}\right).\end{equation}
\end{lemma}
\begin{proof}
That $G$ is continuous and harmonic off $E$ is standard. Since $\mu$ converges weakly to normalized arclength measure on $\Sone$ then (i) holds (for $|z|\leq1$ we have $\int\log|z-x|^{-1}dx=0$ and is equal to $\log|z|^{-1}$ for $|z|>1$).

To prove (ii), fix $x\in\Sone$ and assume the $N$ intervals in the definition of $\mu$ are relabeled so that $x\in I_{0}$ and $I_{1}, I_{2}$ are adjacent to $I_{0}$, and so that $d(x, I_{k})\simeq k/N$ for $k=3,\ldots, N-1$. Let $I=I_{0}\cup I_{1}\cup I_{2}$. Then,
$$ G(x)=\int_{I}\log\frac{1}{|z-x|}d\mu(z)+\sum_{k>2}\int_{I_{k}}\log\frac{1}{|z-x|}d\mu(z)=H_{1}(x)+H_{2}(x).$$
Observe that $H_{1}>0$ on $I$ provided $N$ is large enough (say $N\geq12$). Since $I$ has diameter $\leq1$, we have
\begin{align*}
	\int_{I} H_{1}dx\leq\int_{I}\int_{I}\log\frac{1}{|z-x|}d\mu(z)|dx|\leq\mu(I)\max_{z\in I}\int_{I}\log\frac{1}{|z-x|}|dx|\leq C\frac{\log N}{N^{2}}.
\end{align*}

Therefore, by Tchebyshev's inequality, 
\begin{equation}\label{equation:G1}
 \left| \left\{ x\in I\colon H_{1}(x)\geq \frac{\la}{2}\frac{\log N}{N} \right\} \right| \leq\frac{C}{N\la}.\end{equation}
To estimate $H_{2}$, observe that for $k>2$, the variation of $\log|x-z|^{-1}$ over $I_{k}$ is at most $C/k$. Therefore,
\begin{align}\label{equation:G2}
	\nonumber |H_{2}(x)|\leq&\left| \sum_{k>2}\int_{I_{k}} \log\frac{1}{|z-x|}\left(d\mu(z)-\frac{|dz|}{2\pi}\right) \right|+\left| \sum_{k>2}\int_{I_{k}}\log\frac{1}{|z-x|}\frac{|dz|}{2\pi}\right|\leq \\ 
	\leq& \sum_{k>2}\frac{2\tilde{C}}{kN}+\int_{I}\log\frac{1}{|z-x|}\frac{|dz|}{2\pi}\leq C\frac{\log N}{N},
\end{align}
where we have used that $\int_{\Sone}\log|z-x|^{-1}|dz|=0$ for $x\in\Sone$. Hence, by (\ref{equation:G1}) for $\la=2\delta N/\log N$ and (\ref{equation:G2}), we have
$$ |\left\{ x\in I_{0}\colon |G(x)|\geq\delta \right\}|\leq \frac{C|I_{0}|\log N}{\delta N}.$$

Now, given $J$ an interval, suppose $J=\cup_{k=0}^{j}I_{k}$ with $|I_{k}|\simeq|J|/(j+1)$. Then,
\begin{align*} |\left\{ x\in J\colon |G(x)|\geq\delta \right\}|\leq&\sum_{k=0}^{j}|\left\{x\in I_{k}\colon |G(x)|\geq\delta \right\}\leq \\
\leq& \sum_{k=0}^{j}\frac{C|I_{k}|\log N}{\delta N} = \frac{C|J|\log N}{\delta N}.\end{align*}
Which implies (ii), since the estimates we have obtained are independent of $A$. Finally, since $G_{k}(x)\leq A/N$ and $G_{k}(x)=A/N$ for $x\in I_{k}$, then the estimate on $|H_{2}|$ (\ref{equation:G2}) yields (\ref{estimate_1}). \end{proof}

We can symmetrize $G$ by setting $$U(z)=G(z)+G(z/|z|^{2})=G(z)+G(1/\bar{z}).$$ Then $U$ is harmonic on $\C\setminus (E\cup\{0\})$ and equal to $2G(x)$ on $E$ (and has negative logarithmic poles at $0,\infty$).

\begin{lemma}[\cite{ChrisWeldingAnnals} Lemma 27]\label{lemma:construction_2}
We can choose a (multi-valued) harmonic conjugate $\widetilde{U}$ of $U$ on $\D$ so that $\exp(i\widetilde{U}(x))$ is single valued and $\exp(i\widetilde{U}(x))\to x$ uniformly as $N\to\infty$ (with an estimate independent of $A$).
\end{lemma}
\begin{proof}
$U$ has normal derivative zero on $\Sone\setminus E$, thus $\widetilde{U}$ is constant on each component of $\Sone\setminus E$. On each component $I$ of $E$ we have,
$$ \int_{I}\frac{d\widetilde{U}}{d\theta}d\theta = \int_{I}\frac{dU}{d n}d\theta =\frac{1}{2}\int_{I}\Delta Ud\theta=\int_{I}\Delta Gd\theta=2\pi\mu(I).$$
Therefore, we can choose $\widetilde{U}$ so that $\exp(i\widetilde{U}(x))=x$ at $N$ equidistributed points in $\Sone$. Since $\widetilde{U}$ is monotonic, this implies the result.
\end{proof}

Consider the mapping $\varphi(z)=\exp(U(z)+i\widetilde{U}(z))$. It is conformal (it is a covering and the target is simply connected) and it maps the disk onto a region $W$ which consists of a Jordan domain $\widetilde{W}$ with a finite number of radial slits removed (where $\widetilde{W}$ is a Jordan domain because $\widetilde{U}$ is monotonic on $E$, which in fact implies that $\widetilde{W}$ is star-like with respect to $0$). From the last estimate from Lemma \ref{lemma:construction_1} we see that $\widetilde{W}$ contains the disk $D(0,\exp((2A/N)-1))$ provided $N$ is large enough. Moreover, $W$ contains the disk $D(0,1-C(\log N)/N)$ for some $C<\infty$. This is because, following the notation in the proof of Lemma \ref{lemma:construction_1}, $G(x)=H_{1}(x)+H_{2}(x)$, with $H_{1}>0$ provided $N$ is large enough and $|H_{2}(x)|\leq C(\log N)/N$.

Moreover, by Tchebyshev's inequality, since $||H_{1}||_{1}\leq C(\log N)/N^{2}$, then 
\begin{equation}\label{estimate_2} \left| \left\{ x\in I\colon H_{1}(x)\leq \frac{C}{2}\frac{\log N}{N} \right\} \right| \geq 1/N>0,\end{equation}
(where $I$ is also as in the proof of Lemma \ref{lemma:construction_1}) i.e. there is $x\in I$ so that $U(x)=2G(x)\leq 2C(\log N)/N$. For such $x$, we have $|U(x)|\leq1$ for $N$ large enough and from $|e^{z}-1|\leq e|z|$ if $|z|\leq1$, we obtain
$$ |\varphi(x)|=e^{U(x)}\leq1+e\frac{2C\log N}{N}.$$

Since $\widetilde{U}$ is constant in $\Sone\setminus E$ and for $x\in E$ we have (\ref{estimate_1}), then such $x$ must be near one of the ends of the slits (which gives a bound on how long the slits are). If $A,N$ are large enough, then we can choose a collection of $M$ critical points $\{x_{k}\}\subset\Sone$ of $G$ so that
$$ |x_{k}-\exp(2\pi i k/M)|\leq\delta,$$
(where $M,\delta$ are as in the statement of Proposition \ref{proposition:phi}) and with $$1-\delta\leq\varphi(x_{k})\leq1+\delta$$ for all $k=1,\ldots, M$. Finally, we define $\phi=\varphi/(1+\delta)$.

\begin{figure}[h]
	\includegraphics[scale=0.9]{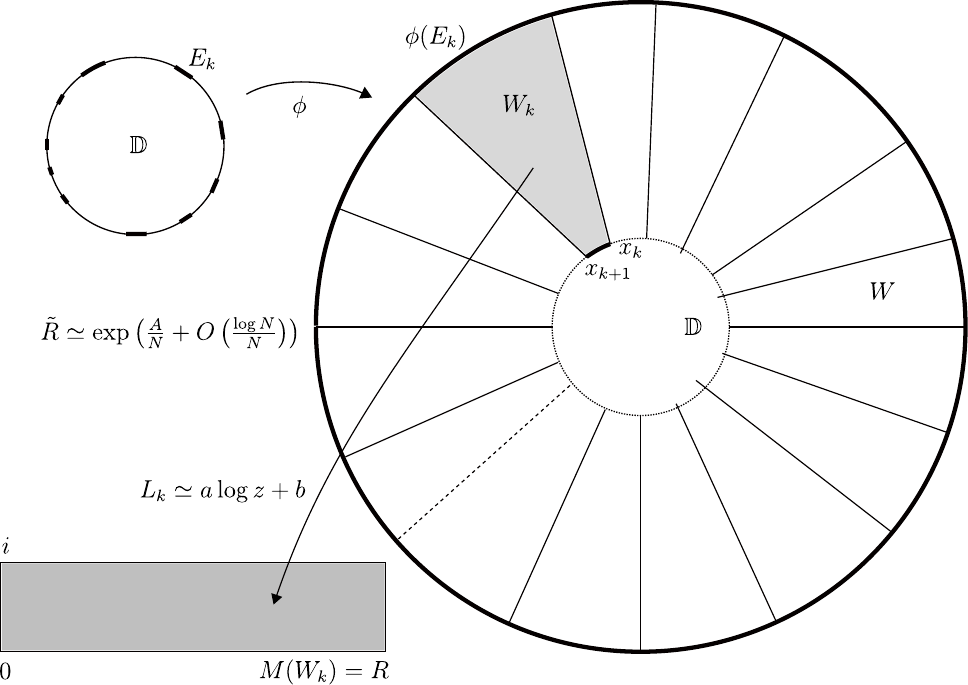}
	\centering
	\caption{Construction of the map $\phi$, as in Proposition \ref{proposition:phi}, which maps the set $E$ as in Lemma \ref{lemma:construction_1} to the boundary of $\widetilde{W}$ and maps $\Sone\setminus E$ to the slits. $W$ consists of $\widetilde{W}$ with the slits removed. The Jordan domain $\widetilde{W}$ is approximately a disk of center $0$ and radius $\tilde{R}\simeq \exp\left(\frac{A}{N}+O\left(\frac{\log N}{N}\right)\right)$.}
	\label{figure:phi_n}
\end{figure}

In fact, (\ref{estimate_1}) yields that $\widetilde{W}$ contains the disk $$ D\left(0,\exp\left(\frac{A}{N}-C\frac{\log N}{N}\right)\right).$$
The slits that end at the points $\{x_{k}\}$ yield quadrilaterals $W_{k}$, as represented in Figure \ref{figure:phi_n}, where we consider the path family off $\D$ that joins two consecutive slits ($W_{k}$ looks like a sector). Observe that we can write an explicit conformal map to a rectangle, which corresponds approximately to a suitable translate of $\log(z)$ (with an appropriate determination of the argument). Therefore it has modulus,
$$M(W_{k})\simeq \frac{2\pi |x_{k+1}-x_{k}|}{\displaystyle\log\frac{\exp\left(\frac{2A}{N}-2C\frac{\log N}{N}\right)}{\exp\left(0-2C\frac{\log N}{N}\right)}}\simeq\frac{\pi N |x_{k+1}-x_{k}|}{A}\simeq\frac{\pi}{A}\frac{N}{M}.$$

Since $|I|\simeq3/N$ it follows from (\ref{estimate_2}) that we can take $M\simeq N$. Hence, $M(W_{k})\simeq\pi/A$ (hence it will be as small as needed by choosing $E$ to have small enough logarithmic capacity).
Therefore, we have that $W_{k}$ can be conformally mapped to a $1\times R_{k}$ rectangle, with $R_{k}\simeq A/\pi$, with $x_{k}, x_{k+1}$ going to the vertices $0,i$ respectively and the radial sides of $W_{k}$ corresponding to the horizontal sides of the rectangle.

This completes the proof of Proposition \ref{proposition:phi}.\qed


\section{Admissible shapes}\label{section:shapes}

As we mentioned in the introduction, in our proof of Theorem \ref{theorem:non-injectivity} we consider embeddings between quadrilaterals, which we need to glue to a given quasiconformal mapping in such a way that the quasiconformality constant stays controlled. In this section we find geometric conditions so that this can be done. We make this explicit by introducing $\eps$-admissible shapes in Definition \ref{definition:admissible_shape}. We provide such geometric conditions in Proposition \ref{proposition:admissible_shape}.

Let $R>T$ and consider the rectangles $[0,R]\times[0,1]$ and $[0,T]\times[0,1]$. If we consider the path families joining opposite horizontal sides of the corresponding rectangles, we have $$M([0,R]\times[0,1])=R>T=M([0,T]\times[0,1]).$$ Throughout this section we will only consider quadrilaterals that have two vertical opposite sides. From now on we will also only consider, without explicit mention, the path family that joins the opposite horizontal sides of the quadrilateral (and when we consider the modulus of one such quadrilateral we will refer to the modulus of that path family). See Figure \ref{figure:embedding}.

\begin{figure}[h]
\includegraphics[scale=1]{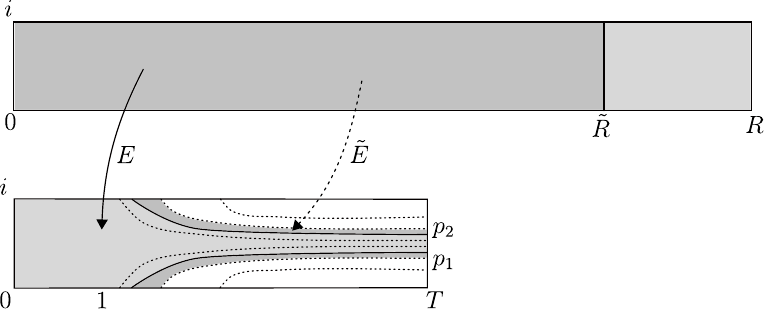}
\centering
\setlength{\unitlength}{\textwidth}
\caption{$E$ maps the rectangle $[0,R]\times[0,1]$ conformally to the lightly shaded subset of the rectangle $[0,T]\times[0,1]$, the points $p_{1}, p_{2}$ represents its two quad-vertices that are contained in $\{T\}\times[0,1]$. $\tilde{E}$ maps the rectangle $[0,\tilde{R}]\times[0,1]$ conformally to the lighter and darker shaded subsets of the rectangle $[0,T]\times[0,1]$. As the distance between the two right most points $p_{1},p_{2}$ decreases the value of $R$ needed increases.}
\label{figure:embedding}
\end{figure}

We proceed to define shapes. Within the rectangle $[0,T]\times[0,1]$ we can define quadrilaterals in the following way: take $0,i$ as vertices and two different points on $\{T\}\times[0,1]$, namely $p_{1}$ and $p_{2}$ with $\mathrm{Im}(p_{1})<\mathrm{Im}(p_{2})$. Consider two non-intersecting simple paths $\gamma_{1},\gamma_{2}\subset[0,T]\times[0,1]$, so that $\gamma_{1}$ connects $0$ to $p_{1}$ and $\gamma_{2}$ connects $i$ to $p_{2}$. This yields a quadrilateral $Q(p_{1}, p_{2},\gamma_{1},\gamma_{2})$ with quad-vertices $0,p_{1}, p_{2}, i$. See Figure \ref{figure:embedding}. The modulus of the quadrilateral $Q(p_{1}, p_{2},\gamma_{1},\gamma_{2})$ depends continuously on our choice of $(p_{1},p_{2},\gamma_{1},\gamma_{2})$, and as long as $T(p_{1}, p_{2},\gamma_{1},\gamma_{2})$ is different than $[0,T]\times[0,1]$, $p_{1}\not=1$ and $p_{2}\not=1+i$ we have $$M(Q(p_{1}, p_{2},\gamma_{1},\gamma_{2}))<M([0,T]\times[0,1]).$$ Since $p_{1}$ is one of the endpoints of $\gamma_{1}$ and $p_{2}$ is one of the endpoints of $\gamma_{2}$, we will omit the dependence on $p_{1}, p_{2}$, i.e. we write $Q(\gamma_{1},\gamma_{2})$.

Let $d(\cdot, \cdot)$ denote the Hausdorff distance between two sets. As $d(\gamma_{1},\gamma_{2})\to0$, we have $M(Q(\gamma_{1},\gamma_{2}))\to\infty$. Since $R>T$, there is a pair $\gamma_{1},\gamma_{2}$ so that $$R=M([0,R]\times[0,1])=M(Q(\gamma_{1},\gamma_{2})).$$ Therefore, there is a conformal map $E\colon [0,R]\times[0,1]\to Q(\gamma_{1},\gamma_{2})$, mapping the vertices of the rectangle to the quad-vertices of the quadrilateral. We say that such $E$ is a \textit{conformal embedding} of $[0,R]\times[0,1]$ into $[0,T]\times[0,1]$, and that $E([0,R]\times[0,1])=Q(\gamma_{1},\gamma_{2})\subset [0,T]\times[0,1]$ is the \textit{shape} of $[0,R]\times[0,1]$. Shapes are not uniquely determined, as we could have made many choices of paths joining the opposite vertical sides of the rectangle $[0,T]\times[0,1]$. This will allow us to effectively control the area of $$\left([0,T]\times[0,1]\right)\setminus E([0,R]\times[0,1])$$ and it will be instrumental when constructing a positive area flexible curve, as well as constructing flexible curves of all other possible Hausdorff dimensions.

Conversely, given a pair $\gamma_{1},\gamma_{2}$ as before, there exists $R>T$ so that $$R=M([0,R]\times[0,1])=M(Q(\gamma_{1},\gamma_{2})).$$

This discussion proves the following.

\begin{lemma}[Shapes]\label{lemma:generateR}
Let $[0,T]\times [0,1]$ be a rectangle and $\gamma_{1},\gamma_{2}$ be non-intersecting simple paths within $[0,T]\times[0,1]$ so that $\gamma_{1}$ joins $0$ to a point $p_{1}\in \{T\}\times[0,1]$, and so that $\gamma_{2}$ joins $i$ to $p_{2}\in\{T\}\times[0,1]$, with $\mathrm{Im}(p_{1})<\mathrm{Im}(p_{2})$. Let $Q(\gamma_{1},\gamma_{2})$ be the quadrilateral defined by the quad-vertices $0,p_{1},p_{2},i$ together with the sides $\gamma_{1},[p_{1},p_{2}],\gamma_{2},[0,i]$. Then there is $R>T$ and a conformal embedding $E$ so that $Q(\gamma_{1},\gamma_{2})=E([0,R]\times[0,1])$.
\end{lemma}

When constructing the quasiconformal map $\psi$ that we mentioned in the introduction in Section \ref{section:construction2} there is a difficulty we need to overcome; we needed to make sure our mapping $E$ from Lemma \ref{lemma:generateR} is close to the identity around the left side of the rectangle $[0,R]\times[0,1]$. See Proposition \ref{proposition:interpolation}. To do so, we need to introduce some quasiconformality in a controlled way.

\begin{definition}[$\eps$-Admissible shape]\label{definition:admissible_shape}
Let $[0,R]\times[0,1]$ and $[0,T]\times[0,1]$ be rectangles with $R>T$. Let $E$ be a conformal embedding of $[0,R]\times[0,1]$ into $[0,T]\times[0,1]$ so that $$[0,1]^{2}\subset E([0,R]\times[0,1]).$$ We say that this shape is $\eps$-admissible if there is a quasiconformal map $$\alpha\colon E([0,R]\times[0,1])\to E([0,R]\times[0,1])$$ fixing the four quad-vertices of the quadrilateral $E([0,R]\times[0,1])$ so that:
\begin{enumerate}[label=(\roman*)]
	\item $\alpha\circ E\colon[0,1]^{2}\to[0,1]^{2}$ is the identity.
	\item The dilatation of $\alpha$ is less than $1+\eps$.
\end{enumerate}
\end{definition}

\begin{figure}[h]
\includegraphics[scale=1]{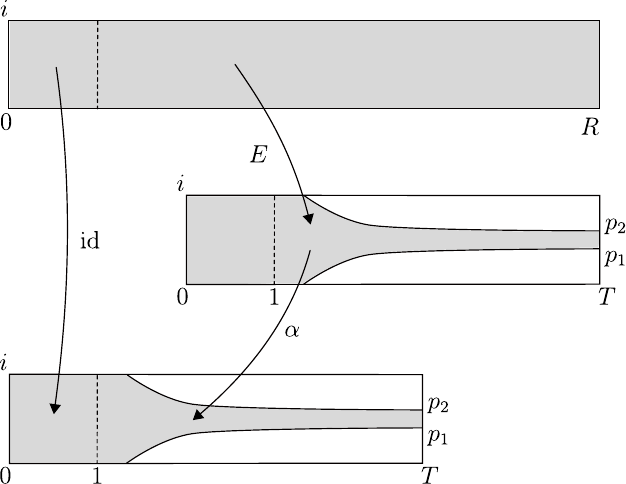}
\centering
\setlength{\unitlength}{\textwidth}
\caption{$E$ maps the rectangle $[0,R]\times[0,1]$ conformally to the lightly shaded subset of the rectangle $[0,T]\times[0,1]$, the points $p_{1}, p_{2}$ represents its two quad-vertices that are contained in $\{T\}\times[0,1]$. The shape $E([0,R]\times[0,1])$ is $\eps$-admissible if there is a quasiconformal map $\alpha\colon E([0,R]\times[0,1])\to E([0,R]\times[0,1])$ with dilatation $\leq\eps$ and so that $\alpha\colon[0,1]\to[0,1]$ is the identity.}
\label{figure:admissible_shape}
\end{figure}

We note that given a $K$-quasiconformal mapping $\alpha$, we have $$||\mu_{\alpha}||_{\infty}=||\alpha_{\bar{z}}/\alpha_{z}||_{\infty}\leq(K-1)/(K+1)$$ and if $||\mu_{\alpha}||_{\infty}\leq k$, then $\alpha$ is $(1+k)/(1-k)$-quasiconformal. If $||\mu_{\alpha}||_{\infty}\leq\eps$, with $\eps>0$ small enough, then $\alpha$ is $\simeq(1+\eps)$-quasiconformal. Conversely, if $\alpha$ is $(1+\eps)$-quasiconformal we have $|| \mu_{\alpha}||_{\infty}\lesssim\eps$. When proving that $\eps$-admissible shapes exist in Proposition \ref{proposition:admissible_shape}, it suffices to show that $||\mu_{\alpha}||_{\infty}\leq\eps$ for a mapping $\alpha$ that we will construct explicitly. See Figure \ref{figure:admissible_shape}.

\begin{proposition}\label{proposition:admissible_shape}
Let $[0,T]\times[0,1]$ be a rectangle with $T\geq6$ and $2\leq M$ a natural number so that $2M<T$. Then any shape $E([0,R]\times[0,1])$ that contains $[0,2M]\times[0,1]$ is $\simeq e^{-\pi(M-1)/2}/M$-admissible. Moreover, the dilatation of the map $\alpha\colon E([0,R]\times[0,1])\to E([0,R]\times[0,1])$ as in Definition \ref{definition:admissible_shape} can be taken to be contained in $[0,M]\times[0,1]$.
\end{proposition}

Let $E([0,R]\times[0,1])$ be a shape so that $$[0,2M]\times[0,1]\subset E([0,R]\times[0,1]).$$ (Use Lemma \ref{lemma:generateR} to find $R$ and the embedding $E$.) To prove Proposition \ref{proposition:admissible_shape} we need to quantify how close $E$ is to the identity in terms of how large $M$ is. To do so, we need to see first how $E([1,1+i])$ differs from $[1,1+i]$ in terms of $M$, and how $E([M,M+i])$ differs from $[M,M+i]$. 

\begin{lemma}\label{lemma:HM_estimate}
Let $[0,2M]\times[0,1]\subset E([0,R]\times[0,1])$ be a shape and define 
$$ \rho_{1}=\min_{z\in[1,1+i]}\mathrm{Re}(E(z)) \quad \textrm{ and }\quad \rho_{2}=\max_{z\in[1,1+i]}\mathrm{Re}(E(z)).$$
Then there exists $C<\infty$, independent of $M$, so that $\rho_{2}-\rho_{1}\leq C e^{-\pi M}$ and $|E'(1+it)-1|\leq C e^{-\pi M}$ for $t\in[0,1]$.
\end{lemma}
\begin{proof}
By reflection we can extend $E$ to a conformal map of the strip $$\{z\in\C\colon-2M\leq\textrm{Re}(z)\leq 2M\}.$$ Observe that the function $v(z)=\textrm{Im}(E(z)-z)$ is harmonic on $$R_{2M}=\{z\colon|\textrm{Re}(z)|\leq 2M, |\textrm{Im}(z)|\leq1|\},$$ is zero on the horizontal sides of $R_{2M}$ and $|\textrm{Im}(E(z)-z)|\leq1$ on the vertical edges of $R_{2M}$ (which we call $E_{2M}$). By the maximum principle, Lemma \ref{lemma:HM_rectangle} and Harnack's inequality we have,
$$ |\textrm{Im}(E(z)-z)|\leq\omega(z,E_{2M},R_{2M})\leq C \omega(0,E_{2M},R_{2M})\leq C_{1} e^{-\pi M}$$
on the rectangle $R_{2M}\cap\{|\textrm{Re}(z)|\leq2\}$. 

By using Herglotz formula for $\D(0,2)$ (see \cite[p. 52]{HarmonicMeasure}), it follows that 
$$ |\nabla v(z)|\leq\left(\int_{-\pi}^{\pi}\frac{1}{|2e^{it}-z|^{2}}\frac{dt}{\pi}\right)||v||_{\infty}\leq\frac{2C_{1} e^{-\pi M}}{(2-|z|)^{2}}\leq \tilde{C} e^{-\pi M}$$
for $|z|\leq\sqrt{2}$. Therefore, by the Cauchy-Riemann equations, $|E'(z)-1|\leq \tilde{C}e^{-\pi M}$. Hence, if $z,w$ have real part $x$,
$$ |\textrm{Re}(E(z))-x+x-\textrm{Re}(E(w))|=|\textrm{Re}(E(z)-z)+\textrm{Re}(w-E(w))|\leq 2\tilde{C} e^{-\pi M}.$$
Therefore,
$$ \rho_{2}-\rho_{1}=\max_{z\in[1,1+i]}\textrm{Re}(E(z))-\min_{z\in[1,1+i]}\textrm{Re}(E(z))\leq 2 \tilde{C} e^{-\pi M}.$$
The constant in the Lemma is $C=2\tilde{C}$. \end{proof}

Given the hypothesis of Lemma \ref{lemma:HM_estimate}, it follows that if we define 
$$ r_{1}=\min_{z\in[M,M+i]}\textrm{Re}(E(z)) \quad \textrm{ and }\quad r_{2}=\max_{z\in[M,M+i]}\textrm{Re}(E(z)),$$
then $r_{2}-r_{1}\leq C e^{-\pi M/2}$ (translate by $-M+1$ and use Lemma \ref{lemma:HM_estimate}). 

Let's make a few observations:
\begin{enumerate}[label=(\alph*)]
	\item By Monotonicity of the module (Proposition \ref{proposition:modulus_properties} (d)), we have $\rho_{1}\leq 1\leq\rho_{2}$, $r_{1}\leq M\leq r_{2}$.
	\item Similarly, $M\simeq r_{1}$.
\end{enumerate}

We can now prove Proposition \ref{proposition:admissible_shape}.

\begin{proof}[Proof of Proposition \ref{proposition:admissible_shape}]
Let's build the quasiconformal map $$\alpha\colon E([0,R]\times[0,1])\to E([0,R]\times[0,1])$$ so that $\alpha\circ E_{|[0,1]^{2}}$ is the identity, as in Figure \ref{figure:strip_argument}. For $a>1$ define the strip $$S_{a}\coloneqq\{1<\textrm{Re}(z)<a\}.$$ By reflection we can extend the map $E$ to the strip $S_{M}$. Our goal is to obtain a quasiconformal map $\alpha$ from the strip $E(S_{M})$ to the strip bounded by $\{\textrm{Re}(z)=1\}$ and $E(M+i\R)$, see Figure \ref{figure:strip_argument}.

We can conformally map this strip to a strip $S_{\tilde{M}}$ via a conformal map $\varphi$. We can build this uniformizing strip $S_{\tilde{M}}$ by uniformizing the quadrilateral within the strip bounded by $\{\textrm{Re}(z)=1\}$ and $E(M+i\R)$ that is contained in $\{|\textrm{Im}(z)|\leq1\}$, and then using reflection to extend it to the rest of the strip.

Observe that $|\tilde{M}-M|\leq 2 C e^{-\pi (M-1)/2}$ by Lemma \ref{lemma:HM_estimate}.

\begin{figure}[h]
\includegraphics[scale=0.9]{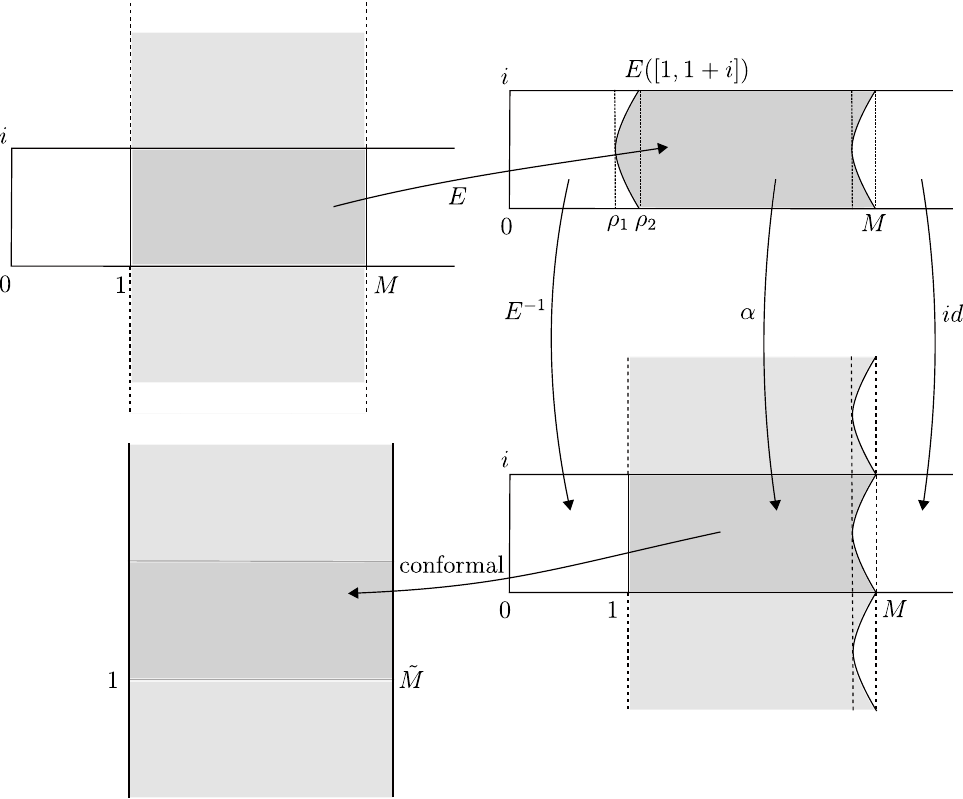}
\centering
\setlength{\unitlength}{\textwidth}
\caption{Extending the maps to a strip. Building the quasiconformal map $\alpha$ by using linear interpolation.}
\label{figure:strip_argument}
\end{figure}

This defines analytic mappings $G_{0},G_{1}$ given by $i G_{0}(t)=\varphi(1+it)-1$ and $i G_{1}(t)=\varphi(E(M+it))-\tilde{M}$. Our goal is to define a quasiconformal extension of these two maps. Define the extension $G\colon S_{M}\to S_{\tilde{M}}$ by 
$$ G(x+iy)=i\frac{x-1}{M-1}G_{1}(y)+i\left(1-\frac{x-1}{M-1}\right)G_{0}(y)+\frac{\tilde{M}-1}{M-1}(x-1)+1.$$
If we write $G=u+iv$, we have, $u_{x}=(\tilde{M}-1)/(M-1)$, $u_{y}=0$, 
$$ v_{x}=\frac{1}{M-1}\left(G_{1}(M+iy)-G_{0}(1+iy)\right)$$
and $$ v_{y}=\frac{x-1}{M-1} G_{1}'(y)+\left(1-\frac{x-1}{M-1}\right)G_{0}'(y).$$
Observe that $G_{1}(y)-G_{0}(y)$ is $1$-periodic and with absolute value bounded by $1$. By our previous estimates, we also have $G_{0}'(y)=\varphi'(1+it)\simeq1$ and $G_{1}'(y)=\varphi' \cdot E'(M+iy)\simeq E'(M+it)\simeq 1$. In fact, we have $|E'-1|\leq Ce^{-\pi M/2}$ and $|\varphi'-1|\leq Ce^{-\pi M/2}$. Hence, 
$$ \left| v_{y}-1\right|=\left| G_{0}'(y)-1+(G_{1}'(y)-G_{0}'(y))\frac{x-1}{M-1}\right|\leq C e^{-\pi M/2}+2C e^{-\pi M/2}.$$

Hence $v_{y}\simeq1$. Therefore, the dilatation of $G$ is given by
$$ \left|\frac{G_{\bar{z}}}{G_{z}}\right|\simeq\left|\frac{-1+\frac{\tilde{M}-1}{M-1}+i\frac{1}{M-1}}{\frac{\tilde{M}-1}{M-1}+1+i\frac{1}{M-1}}\right|\simeq\left|\frac{-1+\frac{\tilde{M}-1}{M-1}}{1+\frac{\tilde{M}-1}{M-1}}\right|\lesssim\left|\frac{M-\tilde{M}}{M-1}\right|\leq \frac{2C e^{-\pi(M-1)/2}}{M-1}.$$

To finish the proof we only need to consider $\alpha=\varphi^{-1}\circ G\circ E^{-1}$. This map can be extended quasiconformally to the rest of $E([0,R]\times[0,1])$, without changing the dilatation, by setting $E^{-1}$ on the left side of $E([1,1+i])$ and the identity on the right side of $E([M,M+i])$.
\end{proof}


\section{Construction of $\psi$}\label{section:construction2}

In this section we construct the map $\psi$ that we mentioned in the introduction. Let $h\colon\Sone\to\Sone$ be a log-singular circle homeomorphism and consider $E\subset\Sone$ and $\phi\colon\D\to W$ conformal as in Proposition \ref{proposition:phi}. Given a pair of $K$-quasiconformal maps $f\colon\D\to\Omega$ and $g\colon\D^{*}\to\Omega_{*}$, where $\Omega, \Omega^{*}$ are smooth Jordan domains with disjoint closures (and with $\infty\in\Omega^{*}$), we want to find a quasiconformal mapping $\psi\colon W\to\widetilde{\Omega}$ so that for $x\in E$ we have that
\begin{equation*} |(\psi\circ\phi)(x)-g(h(x))|\end{equation*} 
can be made small in a sense that will be specified in Section \ref{section:logsingular}. Moreover, we also need to guarantee that the quasiconformality constant of $\psi$ stays as close to the one of $f$ as needed. The map $\psi$ is constructed in Theorem \ref{theorem:psi}.

We will also assume that $f\colon\D\to\Omega$ extends to be quasiconformal (with the same quasiconformality constant) in the disk $D(0,1+\delta)$ for $\delta>0$ small enough.

Recall that in the introduction (Section \ref{section:outline}) we mentioned that the quasiconformal extension of $f$ will be constructed by embedding the quadrilaterals $W_{k}$ from Proposition \ref{proposition:phi} into the annulus $A=\C_{\infty}\setminus(\Omega\cup\Omega^{*})$. To do so, we find a collection of quadrilaterals $\{Q_{k}\}$ so that $Q_{k}\subset A$, and then we find embeddings from $W_{k}$ into $Q_{k}$. From our choice of $Q_{k}$ we will see later in Section \ref{section:logsingular} how bounds on $|(\psi\circ\phi)(x)-g(h(x))|$ follow.

\begin{figure}[h]
	\includegraphics[scale=1]{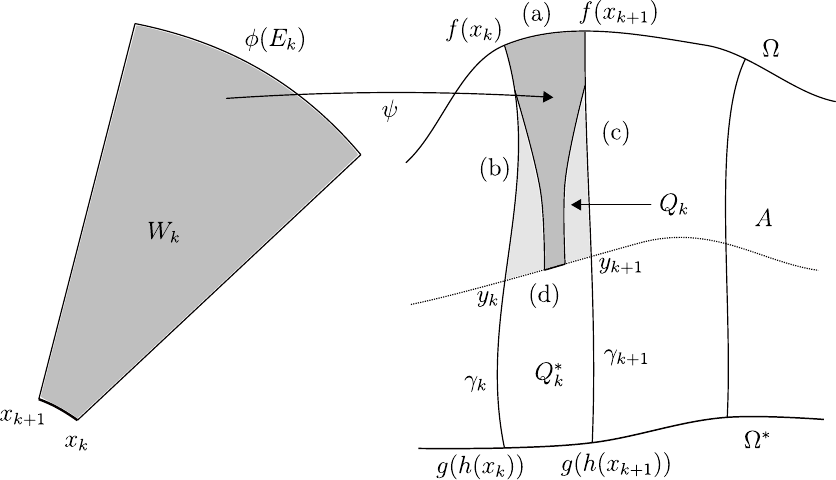}
	\centering
	\caption{Representation of the map $\psi$ from Theorem \ref{theorem:psi}, which is quasiconformal, agrees with $f$ on $[x_{k},x_{k+1}]\subset\Sone$ and satisfies $\psi(W_{k})\subset Q_{k}$. The curves $\gamma_{k}, \gamma_{k+1}$ from the foliation $\mathcal{C}$ are also represented, together with their intersection $y_{k}, y_{k+1}$ with the curve of midpoints $\Gamma$.}
	\label{figure:psi_n}
\end{figure}

Let's define the quadrilaterals $Q_{k}\subset A$. We first define a foliation of $A$. Take the points $\{x_{k}\}$ from Proposition \ref{proposition:phi}. Take the foliation $\mathcal{C}$ of $A$ so that:
\begin{enumerate}[label=(\roman*)]
	\item $\gamma\in\mathcal{C}$ are hyperbolic geodesics $\gamma\subset A$ and are indexed by $t\in[1,N]$ (we write $\gamma_{t}$). Each $\gamma$ has their endpoints on different components of $\partial A$, and it joins some $f(x)\in\partial\Omega$ to $g(h(x))\in\partial\Omega^{*}$.
	\item For $t\not=s$ the geodesics $\gamma_{t}$ and $\gamma_{s}$ are disjoint and $\cup_{t}\gamma_{t}=A$.
	\item $\gamma_{1}$ is the hyperbolic geodesic within $A$ with the shortest Euclidean length that joins $f(x_{1})$ and $g(h(x_{1}))$.
	\item For $k\in\N$ ($1\leq k\leq N$) the geodesic $\gamma_{k}$ joins $f(x_{k})$ and $g(h(x_{k}))$.
\end{enumerate}

See Figure \ref{figure:psi_n} for an illustration. For $1\leq k\leq N$ we consider the midpoint $y_{k}$, with respect to Euclidean length, of the curve $\gamma_{k}$. We now form Jordan quadrilaterals $Q_{k}=Q(x_{k},x_{k+1})$ whose four sides are:
\begin{enumerate}[label=(\alph*)]
	\item The arc on $\partial\Omega$ joining $f(x_{k})$ to $f(x_{k+1})$.
	\item The subarc of $\gamma_{k}$ joining $f(x_{k})$ to $y_{k}$.
	\item The subarc of $\gamma(x_{k+1})$ joining $f(x_{k+1})$ to $y_{k+1}$.
	\item A smooth arc $\sigma_{k}$ joining $y_{k}$ to $y_{k+1}$ so that $\sigma_{k}$ is orthogonal to $\gamma_{k}, \gamma_{k+1}$ and so that it approximates the curve of midpoints of the curves $\gamma_{t}$ for $t\in[k,k+1]$.
\end{enumerate}
The quadrilateral $Q_{k}$ is represented in Figure \ref{figure:psi_n}, which corresponds to the union of the two shaded regions within $A$. Remember that we assumed before that $f$ extends to be quasiconformal, with the same quasiconformality constant, in the $\D(0,1+\delta)$, for $\delta$ small enough. Assume $|x_{k}-x_{k+1}|$ is small enough so that the following conditions hold:
\begin{enumerate}[label=(\roman*)]
	\item $[x_{k},x_{k+1}]\subset\Sone$ is one side of a quadrilateral $\tilde{Q}_{k}$ contained in $\D(0,1+\delta)\setminus\D$ with modulus $M(\tilde{Q}_{k})$, which we will need to assume large enough later (see Figure \ref{figure:psi_n_2}). Observe that $f(\tilde{Q}_{k})$ is a quadrilateral with modulus between $K^{-1} M(\tilde{Q}_{k})$ and $K M(\tilde{Q}_{k})$.
	\item For a fixed $\eta>0$, we can assume $l(\overline{Q_{k}}\cap\sigma_{k})\leq\eta$ provided $N$ (as in Proposition \ref{proposition:phi}) is large enough. That is because $N$ is comparable to the cardinality of the points $\{x_{k}\}$ from Proposition \ref{proposition:phi}, which are approximately uniformly distributed in $\Sone$. 
\end{enumerate}

The main goal of this section is to prove the following, where we use all the notation that we have introduced in this section.

\begin{theorem}[Construction of $\psi$]\label{theorem:psi}
In the same conditions as before, if we define $$V=\D\cup\bigcup_{k}W_{k}\cup\bigcup_{k}[x_{k},x_{k+1}],$$ there is a quasiconformal mapping $$\psi\colon V\to\widetilde{\Omega},$$ so that $\psi_{|\D}=f$, i.e. $\psi$ is a quasiconformal extension of $f$. The map $\psi$ can also be taken to satisfy $\psi(W_{k})\subset Q_{k}$, with the side of $W_{k}$ opposite to $\Sone$ mapping into (but not necessarily onto) the curve $\sigma_{k}$. 

Moreover, given $\eps>0$, then $\psi$ can be taken to have quasiconformality constant $K\cdot(1+\eps)$, provided the logarithmic capacity of $E\subset\Sone$, from Proposition \ref{proposition:phi}, and $\max_{k}|x_{k+1}-x_{k}|$ are both small enough. \end{theorem}

The maps from Proposition \ref{proposition:phi} and Theorem \ref{theorem:psi} are illustrated in Figure \ref{figure:Propositions}.

Notice that in Theorem \ref{theorem:psi} the set $V$ in which $\psi$ is defined consists of the set $W$ from Proposition \ref{proposition:phi} with some slits added back in.

Let $T_{k}$ be the modulus of $Q_{k}$ (where we consider the path family that joins the two sides that are part of the curves $\gamma_{k},\gamma_{k+1}$). Consider $\iota_{k}$ the conformal map from $[0,T_{k}]\times [0,1]$ to $Q_{k}$ that maps quad-vertices to quad-vertices. The value $T_{k}=T(x_{k},x_{k+1})$ depends continuously on the choice of $x_{k},x_{k+1}$ and tends to $\infty$ as $|x_{k}-x_{k+1}|\to0$. Similarly, we can define $T_{k}^{*}$, but for quadrilaterals with vertices $g_{n}(h(x_{k})), g_{n}(h(x_{k+1}))$ and $y_{k}, y_{k+1}$. For each $\delta>0$ we let $$T(\delta)=\max\{T_{k}+T^{*}_{k}\colon |x_{k}-x_{k+1}|\geq\delta\}.$$

We can take $A$ large enough so that $R_{k}\simeq A/\pi>T(1/2M)$, where $R_{k}=M(W_{k})$ as in Proposition \ref{proposition:phi}. $R_{k}$ is approximately the same among all the $W_{k}$, so we write $R_{k}=R$ to simplify.

Since $R>T(\delta)> T_{k}$, there is a conformal map $E_{k}$ from $[0,R]\times [0,1]$ into $[0,T_{k}]\times [0,1]$ with the left side mapping bijectively to the left side and the right side mapping into (but not necessarily onto) the right side. See Figure \ref{figure:embedding}. This map $E_{k}$ is the one constructed in Lemma \ref{lemma:generateR}, where we will always assume that $$[0,1]^{2}\subset E_{k}([0,R]\times [0,1]).$$

This yields the conformal map  $\iota_{k}\circ E_{k}\circ L_{k}$ from our $k$-th radial sector $W_{k}$ into our $k$-th quadrilateral $Q_{k}$. The goal now is to glue the map $\iota_{k}\circ E_{k}\circ L_{k}$ to $f$ along $f([x_{k},x_{k+1}])$ to define $\psi$ (see Figure \ref{figure:Propositions}). To do so, we need to introduce some quasiconformality while we make sure the dilatation is not increased by too much (so that the extension has quasiconformality constant as close to $K$ as we will need later). This procedure is illustrated in the diagram from Figure \ref{figure:diagram}, which we now proceed to explain.

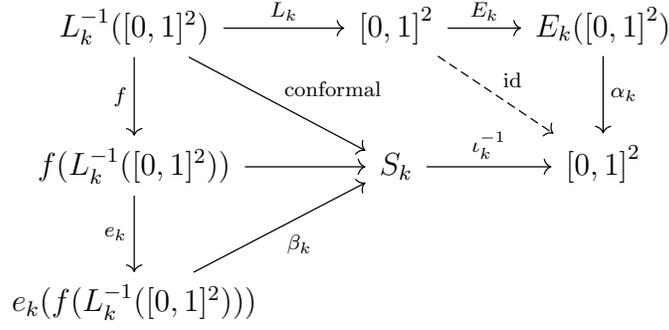
\begin{figure}[h]
\begin{tikzcd}[row sep=2.5em]
L_{k}^{-1}([0,1]^{2}) \arrow[r,"L_{k}"] \arrow[d,"f"'] \arrow[rd,"\textrm{conformal}"]&
\left[0,1\right]^{2} \arrow[r,"E_{k}"] \arrow[rd, dashed, "\textrm{id}"]&
E_{k}(\left[0,1\right]^{2}) \arrow[d,"\alpha_{k}"] \\
f(L_{k}^{-1}([0,1]^{2})) \arrow[r,""] \arrow[d,"e_{k}"'] &
S_{k} \arrow[r,"\iota_{k}^{-1}"]  &
\left[0,1\right]^{2} \\
e_{k}(f(L_{k}^{-1}([0,1]^{2}))) \arrow[ru,"\beta_{k}"']
\end{tikzcd}
\caption{Diagram in the construction of $\psi$. The quad-vertices of the quadrilaterals are mapped to the vertices of the squares. The mappings $L_{k}, E_{k}, \iota_{k}$ are conformal. The mapping $e_{k}$ is conformal, embeds $f(\tilde{Q}_{k})$ into $Q_{k}$, and extends to a conformal map that maps onto $Q_{k}$. The mapping $\beta_{k}$ is so that $\beta_{k}\circ e_{k}\circ f$ is conformal, and it's represented in Figure \ref{figure:psi_n_3}.}
\label{figure:diagram}
\end{figure}

Let $L_{k}$ be a $\C$-linear transformation of the logarithm (with an appropriate determination of the argument) that maps $W_{k}$ to a rectangle $[0,R_{k}]\times[0,1]$. Take $\tilde{Q}_{k}$ as before so that $L_{k}(\tilde{Q}_{k})=[0,M(\tilde{Q}_{k})]\times[0,1]$.

Since $E_{k}$ is conformal, the quadrilateral $E_{k}([0,1]^{2})\subset [0,T_{k}]\times[0,1]$ has modulus $1$. However, by Lemma \ref{lemma:ahlfors_modulus} we will not have $E_{k}([0,1]^{2})=[0,1]^{2}$ unless $E_{k}$ is the identity mapping, but that is not possible since by hypothesis $T_{k}<R_{k}=R$. To fix this we post-compose with a quasiconformal mapping $\alpha_{k}$, as in Section \ref{section:shapes}.

\begin{proposition}[Proposition \ref{proposition:admissible_shape}]\label{proposition:interpolation}
There is a quasiconformal self map $\alpha_{k}$ of $E_{k}([0,R_{k}]\times[0,1])$ so that $\alpha_{k}\circ E_{k}$ is the identity on $[0,1]^{2}$. The quasiconformality constant of $\alpha_{k}$ can be made as small as necessary provided $$[0,M]\times[0,1]\subset E_{k}([0,R_{k}]\times[0,1])$$ for $M$ large enough, and $\alpha_{k}$ can be assumed to have dilatation supported on $[0,M]\times[0,1]$.
In particular, if $T_{k}$ and $R_{k}$ are large enough, then the mapping $E_{k}$ can be taken so that $\alpha_{k}$ has quasiconformality constant as small as needed.
\end{proposition}

Proposition \ref{proposition:interpolation} is just a simplification of Proposition \ref{proposition:admissible_shape}, which we proved in Section \ref{section:shapes}. We emphasize that Proposition \ref{proposition:admissible_shape} also quantifies how small the quasiconformality constant is in terms of $M$, and where it is supported. Controlling the support of the dilatation of the map $\alpha_{k}$ will be needed when generating flexible curves of any possible Hausdorff dimension.

\begin{figure}[h]
	\includegraphics[scale=1]{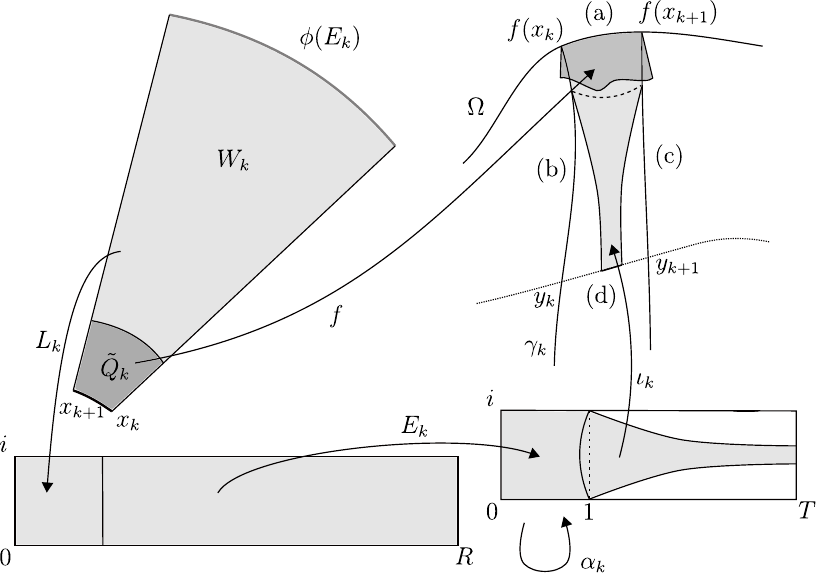}
	\centering
	\caption{The map $\psi$ from Theorem \ref{theorem:psi} is quasiconformal, agrees with $f$ on $[x_{k},x_{k+1}]\subset\Sone$ and $\psi(W_{k})\subset Q_{k}$.}
	\label{figure:psi_n_2}
\end{figure}

Since $f$ can be assumed to be $K$ quasiconformal on $\D(0,1+\delta)$ for $\delta$ small enough, the quadrilateral $f(\tilde{Q}_{k})$ has modulus between $K^{-1} M(\tilde{Q}_{k})$ and $K M(\tilde{Q}_{k})$. In general we will not have $f(\tilde{Q}_{k})\subset Q_{k}$, so we take $$e_{k}\colon f(\tilde{Q}_{k})\hookrightarrow Q_{k}$$ conformal and we write $\tilde{S}_{k}=e_{k}(f(\tilde{Q}_{k}))$. If we choose the maps $e_{k}$ appropriately, we can extend them to conformal maps 
$$ e_{k}\colon e_{k}^{-1}(Q_{k})\to Q_{k},$$
that send quad-vertices to quad-vertices, with $e_{k}^{-1}(Q_{k})\supset f(\tilde{Q}_{k})$ and so that off a neighborhood of $f(\tilde{Q}_{k})$ the quadrilateral $e_{k}^{-1}(Q_{k})$ is exactly $Q_{k}$.

\begin{figure}[h]
	\includegraphics[scale=1]{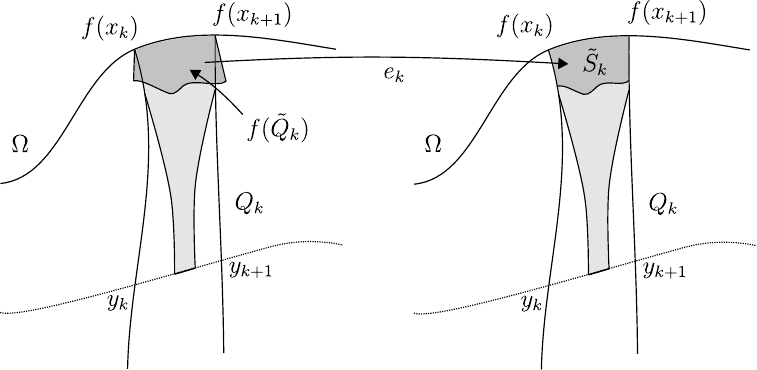}
	\centering
	\caption{Representation of the mapping $e_{k}$. We can take $\beta_{k}\colon Q_{k}\to Q_{k}$ quasiconformal so that the map $\beta_{k}\circ e_{k}\circ f$ is conformal on $\tilde{Q}_{k}$.}
	\label{figure:psi_n_3}
\end{figure}

\begin{proposition}\label{proposition:beta_2}
There exists a quasiconformal map $\beta_{k}\colon Q_{k}\to Q_{k}$ so that 
$$ \beta_{k}\circ e_{k}\circ f \colon L_{k}^{-1}([0,1]^{2})\to S_{k}\coloneqq\iota_{k}([0,1]^{2})$$
is conformal. Moreover, given $\eps>0$ we can assume that $\beta_{k}$ has quasiconformality constant $K\cdot(1+\eps)^{1/2}$ provided $M(\tilde{Q}_{k})$ is large enough, and with its dilatation supported in the quadrilateral $i_{k}([0,M(\tilde{Q}_{k})]\times[0,1]).$
\end{proposition}
\begin{proof}
It suffices to find a mapping $\alpha\colon [0,T_{k}]\times[0,1]\to[0, T_{k}]\times[0,1]$ so that:
\begin{enumerate}[label=(\alph*)]
	\item $ \alpha\circ \iota_{k}^{-1}\circ e_{k}\circ f $ is conformal on $L_{k}^{-1}([0,M(\tilde{Q}_{k})]\times[0,1])$.
	\item $ \alpha\circ \iota_{k}^{-1}\circ e_{k}\circ f(L_{k}^{-1}([0,1]^{2}))=[0,1]^{2}$.
\end{enumerate}
Once we construct such $\alpha$, the map $\beta_{k}$ that we need to find is
$$ \beta_{k}= \iota_{k}\circ \alpha\circ \iota_{k}^{-1}\colon Q_{k}\to Q_{k}.$$
In fact, we will also see that the map $\alpha$ can be taken so that,
\begin{enumerate}[label=(\alph*)]\setcounter{enumi}{2}
	\item The vertices of $[0,1]^{2}$ are fixed under $\alpha\circ \iota_{k}^{-1}\circ e_{k}\circ f \circ L_{k}^{-1}.$ Hence it is the identity map.
\end{enumerate}
First, by making $|x_{k}-x_{k+1}|$ as in the statement of Theorem \ref{theorem:psi} small enough, we can assume that the $K$-quasiconformal map $f$ extends to be $K$-quasiconformal on a quadrilateral $2\tilde{Q}_{k}\subset W_{k}$ of modulus $2M(\tilde{Q}_{k})$ (as we defined before).

There exists a $K\cdot(1+\eps)^{1/4}$-quasiconformal map $$\tilde{\alpha}\colon[0,T_{k}]\times[0,1]\to[0,T_{k}]\times[0,1],$$ so that the map 
$$ E=\tilde{\alpha}\circ \iota_{k}^{-1}\circ \beta_{k}^{1}\circ f\circ L_{k}^{-1} \colon [0, M(\tilde{Q}_{k})]\times[0,1]\to [0,T_{k}]\times[0,1]$$
is conformal, and so that the map $\tilde{\alpha}$ is conformal on $$[0,T_{k}]\times[0,1]\setminus (\iota_{k}^{-1}\circ\beta_{k}^{1}\circ f\circ L_{k}^{-1})([0, 2M(\tilde{Q}_{k})]\times[0,1]).$$
In particular, if $T_{k}$ is large enough, the shape $E([0, M(\tilde{Q}_{k})]\times[0,1])$ is $(1+\eps_{n})^{1/4}$-admissible. Hence by Proposition \ref{proposition:admissible_shape} there is a $(1+\eps_{n})^{1/4}$-quasiconformal map $$\hat{\alpha}\colon E([0, M(\tilde{Q}_{k})]\times[0,1])\to E([0, M(\tilde{Q}_{k})]\times[0,1])$$ so that $\hat{\alpha}\circ E\colon[0,1]^{2}\to[0,1]^{2}$ is the identity. The map $\alpha$ is $\alpha=\hat{\alpha}\circ\tilde{\alpha}$.
\end{proof}

We can now complete the proof of Theorem \ref{theorem:psi}.

\begin{proof}[Proof of Theorem \ref{theorem:psi}]
Take $\beta_{k}\colon Q_{k}\to Q_{k}$ quasiconformal as in Proposition \ref{proposition:beta_2}, and with $M(\tilde{Q}_{k})$ large enough so that the dilatation of $\beta_{k}$ is bounded by $$K\cdot(1+\eps)^{1/2}.$$

The map $\iota_{k}^{-1}\colon Q_{k}\to [0,T_{k}]\times[0,1]$ is conformal and maps $S_{k}$ to $[0,1]^{2}$. The diagram from Figure \ref{figure:diagram} is not commutative. However, the mapping
$$ \iota_{k}^{-1}\circ \beta_{k}\circ e_{k}\circ f\circ L_{k}^{-1}\colon[0,1]^{2}\to[0,1]^{2} $$
is conformal and the identity on the vertices. Thus it is the identity map. Hence,
$$ {f}_{|L_{k}^{-1}([0,1]^{2})}=e_{k}^{-1}\circ\beta_{k}^{-1}\circ\iota_{k}\circ {L_{k}}_{|L_{k}^{-1}([0,1]^{2})}.$$
Since $\alpha_{k}\circ E_{k}$ is the identity on $[0,1]^{2}$, the quasiconformal mapping 
$${\psi}_{|W_{k}}=e_{k}^{-1}\circ\beta_{k}^{-1}\circ \iota_{k}\circ\alpha_{k}\circ E_{k}\circ {L_{k}}_{| W_{k}}$$
defines a quasiconformal extension of $f$ on $W_{k}$ and maps the opposite side of $[x_{k},x_{k+1}]\subset\Sone$ from the quadrilateral $W_{k}$ into the opposite side of $f([x_{k},x_{k+1}])$ within $Q_{k}$. Moreover, within $\cup_{k} \psi(W_{k})$ the map $\psi$ is conformal off $\cup_{k} \psi(\tilde{Q}_{k})$.

The quasiconformality constant of $\psi$ depends on the quasiconformality constant of $f, \alpha_{k}$ and $\beta_{k}$. By Proposition \ref{proposition:interpolation} (or Proposition \ref{proposition:admissible_shape}) if we choose the set $E([0,R_{k}]\times[0,1])$ appropriately, we can make sure the quasiconformality constant of $\alpha_{k}$ is bounded by $ (1+\eps)^{1/2}.$ Thus, ${\psi}_{|W_{k}}$ has quasiconformality constant bounded by $$ K\cdot(1+\eps)^{1/2}\cdot(1+\eps)^{1/2}= K\cdot (1+\eps),$$
as we wanted to see. \end{proof}


\section{Log-singular circle homeomorphisms are weldings}\label{section:logsingular}

In this section we review the proof of \cite[Theorem 3]{ChrisWeldingAnnals}, where Bishop showed that a circle homeomorphism is log-singular if and only if it is the welding of a flexible curve. The main ingredients of the proof are the map $\phi$ from Section \ref{section:construction1} (Proposition \ref{proposition:phi}) and the map $\psi$ from Section \ref{section:construction2} (Theorem \ref{theorem:psi}). By introducing admissible shapes in Section \ref{section:shapes}, we have in fact already significantly expanded Bishop's proof.  We also sharpen his arguments, which ultimately allows us to produce positive area flexible curves, as well as flexible curves will all other possible Hausdorff dimensions.

As in Bishop's paper, we start with the easy direction, which follows from a theorem of Balogh and Bonk \cite{BaloghBonk}.

\begin{lemma}[\cite{ChrisWeldingAnnals} Lemma 24]\label{lemma:easy_flexible}
Suppose $h$ is a conformal welding associated to a flexible curve. Then there is a set $E\subset\Sone$ such that both $E$ and $h(\Sone\setminus E)$ have zero capacity, i.e. $h$ is log-singular.
\end{lemma}
\begin{proof} Suppose $n$ is large enough and let $\Gamma_{n}=\partial W_{n}$, where $W=[-1,n^{2}]\times[-1,1]$. Since $h$ is the welding of a flexible curve, there is a curve $\gamma_{n}$ corresponding to $h$ which lies within Hausdorff distance $1/4$ of $\Gamma_{n}$. Hence, there are conformal maps $f\colon\D\to\Omega_{n}$, $g\colon\C\setminus\overline{\D}\to\C\setminus\overline{\Omega_{n}}$, where $\partial\Omega_{n}=\gamma_{n}$ and $\Omega_{n}$ is bounded. We can further assume that $f(0)=0$.

Let $ E_{n}=f^{-1}(\{x+iy\in\Gamma_{n}\colon x\geq n\}).$ Then 
$$ h(\Sone\setminus E_{n})=g^{-1}(\{x+iy\in\Gamma_{n}\colon x<n\}).$$
By Lemma \ref{lemma:capacity_far} both $E_{n}$ and $h(\Sone\setminus E_{n})$ have small logarithmic capacity. Therefore by taking $n\to\infty$ and Lemma \ref{lemma:equivalence_logsingular}, it follows that $h$ is log-singular.
\end{proof}

The goal for the remainder of this section is to prove the opposite direction.

\begin{theorem}[\cite{ChrisWeldingAnnals} Theorem 25]\label{theorem:log-singular_to_flexible}
Suppose $h$ is an orientation-preserving log-singular homeomorphism. Suppose that there are conformal maps $f_{0}\colon\D\to\Omega$ and $g_{0}\colon\D^{*}\to\Omega^{*}$ onto Jordan domains with disjoint closures such that $\infty\in g_{0}(\D^{*})$. Then for any $r<1$ and any $\eta>0$, there are conformal maps $F$ and $G$ of $\D$ and $\D^{*}$ onto the two complementary components of a Jordan curve $\Gamma$ such that $h=G^{-1}\circ F$ on $\Sone$, $|F(z)-f_{0}(z)|\leq\eta$ for all $|z|\leq r$ and $|G(z)-g_{0}(z)|\leq\eta$ for all $|z|\geq1/r$.
\end{theorem}

Which together with Lemma \ref{lemma:easy_flexible} implies \cite[Theorem 3]{ChrisWeldingAnnals}. We will actually prove the following apparently weaker statement, which turns out to be sufficient to prove Theorem \ref{theorem:log-singular_to_flexible} by Proposition \ref{proposition:qcw}.

\begin{theorem}\label{theorem:log-singular_welding}
Let $h$ be an orientation-preserving log-singular homeomorphism and $f_{0}\colon\D\to\Omega$, $g_{0}\colon\D^{*}\to\Omega^{*}$ be conformal maps onto domains with disjoint closure and smooth boundary so that $\infty\in g_{0}(\D^{*})$. Fix a sequence $\{\eps_{n}\}_{n\in\N}$ of positive real numbers so that if $K_{n}=\prod_{j=1}^{n}(1+\eps_{j})$, then
$K=\prod_{n}(1+\eps_{n})=\lim_{n} K_{n}<\infty$ (assume $\eps_{0}=0$). Then there are $K_{n}$-quasiconformal maps $f_{n}\colon\D\to\Omega_{n}$, $g_{n}\colon\D^{*}\to\Omega_{n}^{*}$ onto smooth Jordan domains with disjoint closures so that if the annulus $A_{n}=\C\setminus(\Omega_{n}\cup\Omega_{n}^{*})$ is foliated by a family $\mathcal{C}_{n}$ as in Theorem \ref{theorem:psi}, that we now index by $x\in\Sone$, we have:
\begin{enumerate}[label=(\roman*)]
	\item $l(\gamma_{n+1}(x))\leq(2/3)l(\gamma_{n}(x))$, where $l(\gamma)$ denotes the length of $\gamma$. Hence,
	$$ |f_{n}(x)-g_{n}(h(x))|\leq (2/3)^{n}L.$$
	\item $|f_{n+1}(x)-f_{n}(x)|\leq C\left(2/3\right)^{n} L$ and $|g_{n+1}(x)-g_{n}(x)|\leq C\left(2/3\right)^{n} L$.
\end{enumerate}
Therefore, by the maximum modulus principle the sequences $f_{n}, g_{n}$ converge to $K$-quasiconformal maps $f\colon\D\to\Omega_{\infty}, g\colon\D^{*}\to\Omega_{\infty}^{*}$ onto the complementary components of a Jordan curve $\Gamma$ so that $h=g^{-1}\circ f$ on $\Sone$.
\end{theorem}

\begin{figure}[h]
	\includegraphics[scale=1]{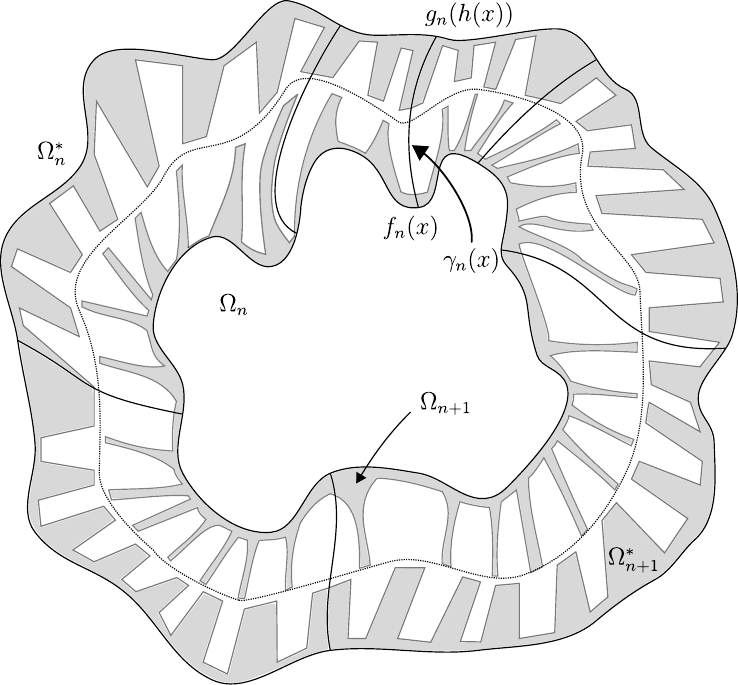}
	\centering
	\caption{Representation of the iterative procedure in Theorem \ref{theorem:log-singular_welding}. The domains $\Omega_{n+1}, \Omega_{n+1}^{*}$ are constructed from $\Omega_{n}, \Omega_{n}^{*}$ by using a quasiconformal extension of $f_{n}, g_{n}$ in such a way that the sequence $||f_{n}(x)-g_{n}(h(x))||_{\infty}$ decreases with $n$.}
	\label{figure:ChrisMainTheorem}
\end{figure}

Before we proceed to prove Theorem \ref{theorem:log-singular_welding}, let's review how it implies Theorem \ref{theorem:log-singular_to_flexible}. Take a quasiconformal map $H\colon\C\to\C$ so that $F\coloneqq H\circ f\colon\D\to\phi(\Omega_{\infty})$ and $G\coloneqq H\circ g\colon \D^{*}\to\phi(\Omega_{\infty}^{*})$ are conformal. The Jordan curve $H(\Gamma)$ has conformal welding $h$. Let $r<1$, by compactness of quasiconformal mappings (or Theorem \ref{theorem:Ahlfors_formula}), if the sequence $\{\eps_{n}\}$ in Theorem \ref{theorem:log-singular_welding} is chosen so that $K=\prod(1+\eps_{n})$ is close enough to $1$, the estimates from Theorem \ref{theorem:log-singular_to_flexible} will follow.

As we mentioned before, we now use the results from Section \ref{section:construction1} and Section \ref{section:construction2} to prove Theorem \ref{theorem:log-singular_welding}. See Figure \ref{figure:Propositions}. All the constructions we make for the sequence $f_{n}$ will analogously work for $g_{n}$.

\begin{figure}[h]
	\includegraphics[scale=0.85]{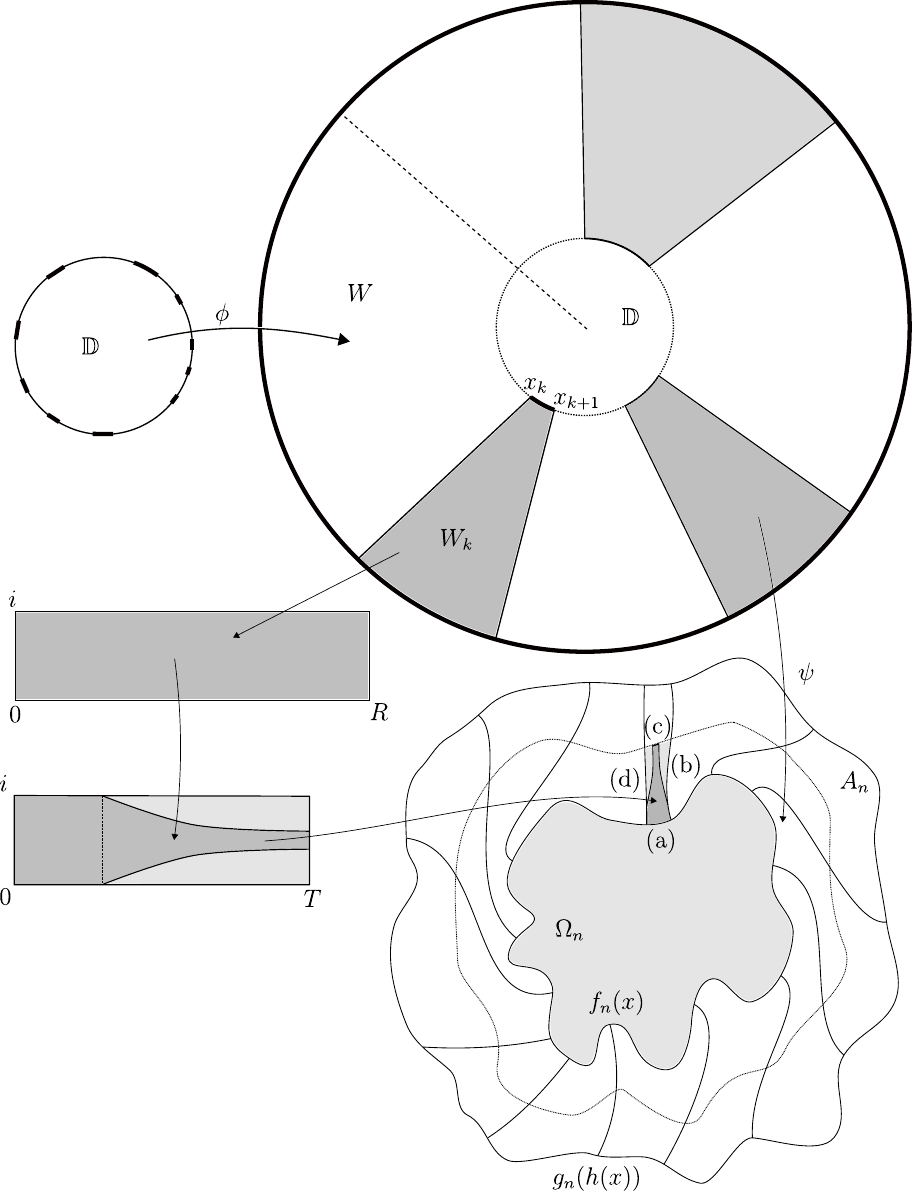}
	\centering
	\caption{Representations of the map $\phi$ from Proposition \ref{proposition:phi} and the map $\psi$ from Theorem \ref{theorem:psi}. The map $f_{n+1}$ from Theorem \ref{theorem:log-singular_welding} is roughly $\psi\circ\phi$. }
	\label{figure:Propositions}
\end{figure}

For the base case $n=0$ we can assume the that the map $f_{0}$ extends conformally to a disk $\D(0,1+\delta)$ by either assuming that $\Omega_{0}$ has analytic boundary or by considering $f_{0}(tz)$ for $t$ close enough to $1$.

Suppose we have obtained mappings $f_{0}, f_{1}, \ldots, f_{n}$ (and $g_{0}, g_{1}, \ldots, g_{n}$) that satisfy the conclusions of Theorem \ref{theorem:log-singular_welding}. We use Theorem \ref{theorem:psi} with $f_{n}$, $g_{n}$ and $\eps_{n}$, which at the same time also generates $\phi$ from Proposition \ref{proposition:phi}. The map $\psi\circ\phi$ is $K_{n}(1+\eps_{n})=K_{n+1}$-quasiconformal, close to $f_{n}$ on compact subsets of $\D$ and is closer to $g_{n}\circ h$ on a set $E_{n}\subset\Sone$ of very small logarithmic capacity so that $h(\Sone\setminus E_{n})$ also has very small logarithmic capacity (see Lemma \ref{lemma:construct_E} and Proposition \ref{proposition:phi}). However, $(\psi\circ\phi)(\D)$ is not a Jordan domain. To address this, we will define $f_{n+1}=(\psi\circ\phi)(tz)$ for $t$ close enough to $1$. 

Similarly, we can construct a mapping $g_{n+1}$ that is quasiconformal, that is close to $g_{n}$ on compact subsets of $\D^{*}$. If we take $t$ close enough to $1$, then it follows that for $x\in E_{n}$ the geodesic joining $f_{n+1}(x)$ and $g_{n+1}(h(x))$ has now length $\leq(\eps+1/2) l(\gamma(x))$ (where $\gamma$ was the geodesic joining $f_{n}(x)$ to $g_{n}(h(x))$). This is justified by the following lemma, in which both the statement and proof are the same as in Bishop's paper.

\begin{lemma}[\cite{ChrisWeldingAnnals} Lemma 28]\label{lemma:control_approx}
Suppose $\Gamma$ is a smooth closed Jordan curve and $\Omega, \Omega^{*}$ are simply connected domains obtained from the complementary components of $\Gamma$ by removing a finite number of smooth, disjoint arcs $\{\gamma_{k}\}$, each disjoint from $\Gamma$ for except exactly one endpoint on $\Gamma$. Let $f\colon\D\to\Omega$, $g\colon\D^{*}\to\Omega^{*}$ be conformal maps and let $A_{t}=\C\setminus(f(t\D)\cup g(\D^{*}/t))$. Then for any $\eps>0$ there is a $t<1$ so that the hyperbolic geodesic in $A_{t}$ connecting the boundary points $f(tx)$ and $g(y/t)$ has Euclidean length at most $L(x,y)+\eps$, where $L(x,y)$ is the length of the shortest path connecting $f(x)$ and $g(y)$ in $X=\partial\Omega\cup\partial\Omega^{*}$.
\end{lemma}
\begin{proof}
If $t$ is close enough to $1$ then $A_{t}$ locally looks like a strip for except at the finitely many tips of the arcs $\gamma_{k}$ and the finitely many points where they intersect $\Gamma$. There $M=2\cdot\#\{\gamma_{k}\}$ tips and intersection points. Cover each one of them by a $\eps/(2M)$ disk. Given a geodesic in $A_{t}$, then inside of these disks, by the Gehring-Hayman inequality, the distance is less than $C\eps/M$ (the diameter of the disk). Away from the disks, Euclidean length of the geodesic in $A_{t}$ is close to the length on $X$.
\end{proof}

This lemma guarantees that when going from step $n$ to step $n+1$ we can maintain control on the Euclidean length of the geodesics on our new foliation.

We now complete the proof of Theorem \ref{theorem:log-singular_welding}.

\begin{proof}[Proof of Theorem \ref{theorem:log-singular_welding}] We have already verified the base case before. Suppose we have verified the induction hypothesis up to $n\in\N$. Set $f_{n+1}$ as before, which can be taken to be $K_{n+1}$-quasiconformal by Theorem \ref{theorem:psi}. If $\eta$ is small enough in Theorem \ref{theorem:psi} we have that on $E_{n}$, as in Proposition \ref{proposition:phi}, 
$$ l(\gamma_{n+1}(x))\leq 2/3 l(\gamma_{n}(x)),$$
where $\gamma_{m}(x)$ denotes the hyperbolic geodesic joining $f_{m}(x)$ to $g_{m}(h(x))$.

Similarly, we obtain $g_{n+1}$ in such a way that for $y\in h(\Sone\setminus E_{n})$, we have
$$ l(\gamma_{n+1}(h^{-1}(y)))\leq 2/3 l(\gamma_{n}(h^{-1}(y))).$$
Since $\Sone=E_{n}\cup h^{-1}(h(\Sone\setminus E_{n}))$, we obtain (i) from Theorem \ref{theorem:log-singular_welding}.

(ii) holds for similar reasons since $\psi_{n}$ is a quasiconformal extension of $f_{n}$ (and analogously for $g_{n+1}$).
\end{proof}


\section{Positive area flexible curves}\label{section:positive_area}

In this section, we prove that there are positive area flexible curves by using $\eps$-admissible shapes. To do so, we refine the construction used to prove Theorem \ref{theorem:log-singular_welding} by controlling, when constructing the mapping $\psi_{n}$ from Theorem \ref{theorem:psi}, how much percentage of area the embedding $E_{k}$ leaves behind.

We emphasize that when proving Theorem \ref{theorem:log-singular_welding}, at each step from the inductive construction the value of $R$ (which is approximately the same across all $W_{k}$ by Proposition \ref{proposition:phi}) can be taken to be as large as needed. The value $T_{k}$ could potentially change, even though we can always assume it is larger than a fixed large value.

Let $\eps>0$ and fix a rectangle $[0,T]\times[0,1]$. We have seen in Proposition \ref{proposition:admissible_shape} that there are $\eps$-admissible shapes provided $T$ is large enough. 

\begin{definition}[Leftover percentage of a shape]
Given an $\eps$-admissible shape $E$ associated with the rectangle $[0,R]\times[0,1]$. We call the quantity $$a(E, T)=\frac{m(([0,T]\times[0,1])\setminus E([0,R]\times[0,1]))}{m([0,T]\times[0,1])}\in[0,1)$$ the \textit{leftover percentage} of the shape, as represented in Figure \ref{figure:leftover_area}, where $m$ denotes Lebesgue measure.
\end{definition}

\begin{figure}[h]
	\includegraphics[scale=1]{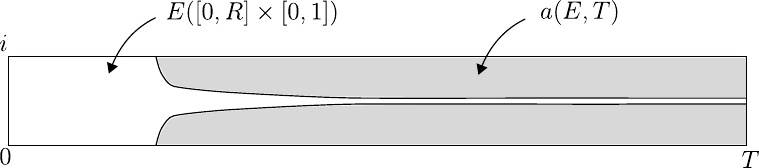}
	\centering
	\caption{Representation of the percentage of area left by a shape with respect to the total area of the rectangle $[0,T]\times[0,1]$ it is embedded into. The \textit{percentage of area left} is represented by the light shade. In this case $a(E,T)>0.75$.}
	\label{figure:leftover_area}
\end{figure}

In Theorem \ref{theorem:positive_area} we will construct a flexible curve by taking an infinite intersection of sets of positive area. If we make sure the area does not decrease by too much step by step, we will be able to guarantee that the set we obtain also has positive area (which will be a curve by Theorem \ref{theorem:log-singular_welding}). This is analogous to when we construct Cantor sets of positive length.

Observe that by the estimates in Proposition \ref{proposition:admissible_shape}, we need our $T$ to be large enough so that there is an $M$ with $1<M<2M<T$ and,
$$ \frac{e^{-\pi (M-1)/2}}{M-1}=\left(\exp(\pi (M-1)/2 +\log (M-1))\right)^{-1}<\eps.$$
That is,
\begin{equation}\label{equation:M_and_eps}
 \log(1/\eps)<\pi\frac{M-1}{2}+\log (M-1). \end{equation}
In particular, we can take $M>1+2\log(1/\eps)/\pi$.

\begin{lemma}\label{lemma:area_left}
Let $\eps>0$ and $a\in(0,1)$, then there are rectangles $[0,T]\times[0,1]$ and $[0,R]\times[0,1]$, with $R>T$, that admit an $\eps$-admissible shape $E$ so that $a(E,T)=a$.
\end{lemma}
\begin{proof}
Given $\eps>0$, choose $M>1$ satisfying (\ref{equation:M_and_eps}). Now take $T>2M$ large enough so that $$2M/T<(1-a)/2.$$ Now we can find a shape $E$ with $a(E,T)=a$ that is $\eps$-admissible (Proposition \ref{proposition:admissible_shape}). Lemma \ref{lemma:generateR} yields $E=E([0,R]\times[0,1])$ for some $R$, which can be taken to be as large as needed by modifying the quad-vertices on the right side of $E([0,R]\times[0,1])$.
\end{proof}

Observe that in the proof of Lemma \ref{lemma:area_left} we needed $4M/(1-a)<T$. The goal now is to apply Lemma \ref{lemma:area_left} to a collection of $T_{j}$'s and obtain a consistent $R$ among them.

\begin{proposition}\label{proposition:area_left}
Let $\eps>0$, $0<a<1$ and suppose $M$ is taken so that (\ref{equation:M_and_eps}) holds. If $\{T_{j}\}_{j=1}^{n}\subset\N$ satisfies $4M/(1-a)<T_{j}$ for $j=1,\ldots,n$, then we can find $R<\infty$ and $\eps$-admissible shapes $E_{j}([0,R]\times[0,1])\subset [0,T_{j}]\times[0,1]$ so that $a(E_{j},T_{j})=a$.
\end{proposition}
\begin{proof}
From Lemma \ref{lemma:area_left} it follows that we can find shapes $$\tilde{E}_{j}=\tilde{E}_{j}([0,R_{j}]\times[0,1])\subset [0,T_{j}]\times[0,1]$$ that are $\eps$-admissible and so that $a(\tilde{E}_{j},T_{j})=a$. If we take $R=\max R_{j}$, then we can modify the quad-vertices on the right side of the quadrilateral $\tilde{E}_{j}$ by using Lemma \ref{lemma:generateR} so that $E_{j}=E_{j}([0,R]\times[0,1])$ and with $E_{j}$ being the same subset of $[0,T_{j}]\times[0,1]$ as $\tilde{E}_{j}$.
\end{proof}

We can now prove that there are positive area flexible curves, which proves Theorem \ref{theorem:non-injectivity} for $s=2$.

\begin{theorem}[Positive area flexible curves]\label{theorem:positive_area}
Let $h\colon\Sone\to\Sone$ be an orientation-preserving log-singular circle homeomorphism and $\{\eps_{n}\}_{\N}$ a sequence of positive real numbers so that if $K_{n}=\prod_{j=1}^{n}(1+\eps_{j})$, then $K=\lim_{n}K_{n}<\infty$. Consider $\{a_{n}\}_{n}\subset(0,1)$ so that $\prod_{n\in\N}a_{n}>0$. Then there are $K$-quasiconformal maps $f\colon\D\to\Omega$, $g\colon\D^{*}\to\Omega^{*}$ so that $\Omega,\Omega^{*}$ are disjoint domains, $J=\partial\Omega=\partial\Omega^{*}$ is a Jordan curve satisfying  $h=g^{-1}\circ f$ and so that
$$ m(J)=\prod_{n\in\N} a_{n}>0.$$ 
Therefore, there is a positive area flexible curve with welding $h$.
\end{theorem}
\begin{proof}
We construct a sequence of annuli $A_{n}$ as in Proposition \ref{proposition:phi} and Theorem \ref{theorem:psi} so that
$$ \liminf_{n}m(A_{n})\geq m(A_{0})\prod_{n\in\N}a_{n},$$
where $m$ denotes Lebesgue measure. To do so, it suffices to see that sequence of annulus $\{A_{n}\}_{n}$ can be taken to satisfy $m(A_{n})=a_{n} m(A_{n-1})$ for $n\in\N$. See Figure \ref{figure:positive_area_curve} for an illustration.

\begin{figure}[h]
	\includegraphics[scale=1]{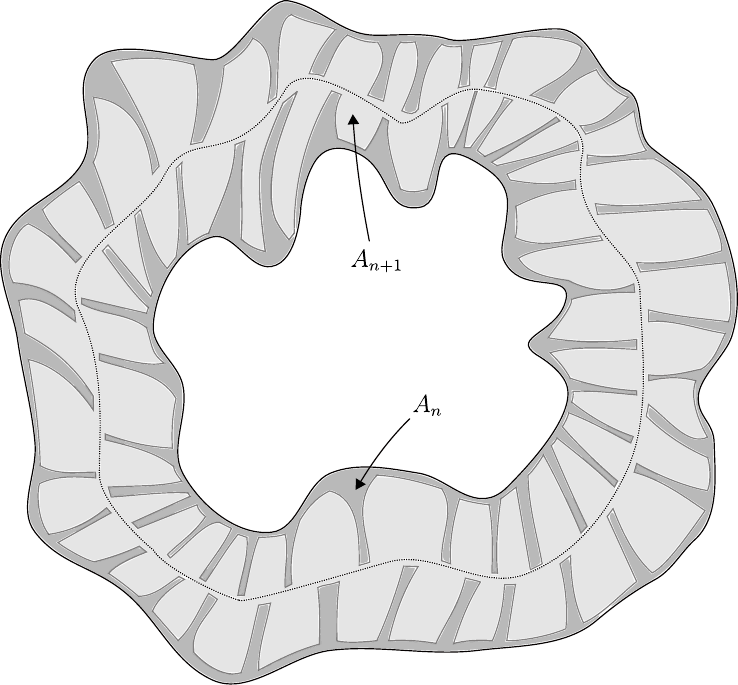}
	\centering
	\caption{Construction of a positive area flexible curve with welding $h\colon\Sone\to\Sone$, a fixed log-singular homeomorphism. $A_{n+1}$ takes a big percentage of the area of $A_{n}$. $A_{n+1}$ corresponds to the annulus that has a lighter color. $A_{n}$ is $\overline{A_{n+1}}$ together with the two annulus of darker color.}
	\label{figure:positive_area_curve}
\end{figure}

For then, when we apply Lemma \ref{lemma:control_approx} in the proof of Theorem \ref{theorem:log-singular_welding}, if we make the $\eps$ there smaller at each step in a suitable way, the resulting Jordan curve $J$ will have area $$ m(A_{0})\geq m(J)\geq\frac{m(A_{0})}{2}\prod_{n\in\N}a_{n}>0.$$ 
By re-scaling by a suitable conformal map of $\C$ we can obtain $m(J)=\prod_{n}a_{n}>0$. We find now a quasiconformal map $H\colon\C\to\C$ with zero dilatation over $J$ so that $H^{-1}\circ f$ and $H^{-1}\circ g$ are conformal. Then,
$$ h=g^{-1}\circ f_{|\Sone}=(H^{-1}\circ g)^{-1}\circ (H^{-1}\circ f)_{|\Sone}.$$
Since quasiconformal maps preserve sets of zero area, it follows that $H^{-1}(J)$ is a positive area flexible curve.

We proceed to explain how to construct such sequence of annuli, which is illustrated in Figure \ref{figure:positive_area_curve}. When building the quasiconformal mapping $f_{n}$ at step $n>1$ in the proof of Theorem \ref{theorem:log-singular_to_flexible}, we had to make $|x_{k}-x_{k+1}|$ small for a number of reasons, when doing so the value of the modulus of the quadrilaterals $Q_{k}$ (see Figure \ref{figure:psi_n}, Figure \ref{figure:diagram} and Figure \ref{figure:psi_n_2}), which we call $T_{k}$, increases. We take $T_{k}$ to be large enough so that the hypothesis from Proposition \ref{proposition:area_left} are satisfied for $\eps_{n}$ and $\tilde{a}_{n}$ (which we specify later).

\begin{figure}[h]
	\includegraphics[scale=1]{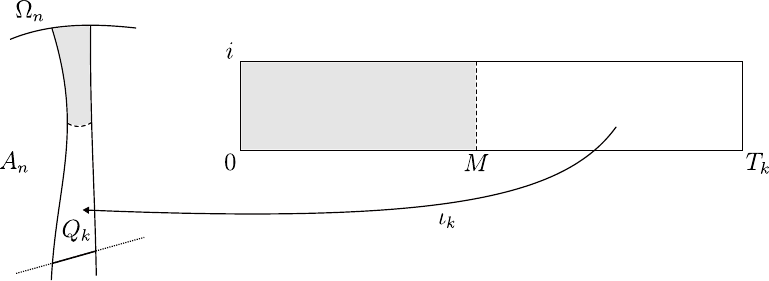}
	\centering
	\caption{Representation of $\iota_{k}\colon[0,T_{k}]\times[0,1]\to Q_{k}$, which is conformal and sends vertices to quad-vertices, with the vertices $0,i$ corresponding to the quad-vertices of $Q_{k}$ contained in $\partial\Omega_{n}$. The dashed segment $[M,M+i]$ is mapped to the dashed curve in $Q_{k}$ via $\iota_{k}$. The percentage $m(\iota_{k}([0,M]\times[0,1]))/m(Q_{k})$ is comparable to $M/T_{k}$.}
	\label{figure:positive_area_proof}
\end{figure}

Consider the conformal map $\iota_{k}$ so that $$\iota_{k}([0,T_{k}]\times[0,1])=Q_{k},$$ mapping vertices to quad-vertices, where $Q_{k}$ is defined in Theorem \ref{theorem:psi}. See Figure \ref{figure:diagram}, Figure \ref{figure:psi_n_2} and Figure \ref{figure:positive_area_proof}. The map $\iota_{k}$ changes area, but the $M$ as in Proposition \ref{proposition:area_left} needed to obtain $\eps_{n}$-admissible shapes only depends on $\eps_{n}$ (see (\ref{equation:M_and_eps})). Modulus estimates justify that $M/T_{k}$ is comparable to $$m(\iota_{k}([0,M]\times[0,1]))/m(Q_{k})$$ by a factor $L$ that only depends on $M$ (and not on $T_{k}$). 

It follows that we can find $\tilde{a}_{k}\in(0,1)$, in terms of $a_{k}$ and $L$, and apply Proposition \ref{proposition:area_left} with $\eps_{n},\tilde{a}_{n}$ to obtain an $\eps_{n}$-admissible shape $E_{k}\subset[0,T_{k}]\times[0,1]$ so that $a(E_{j}, T_{k})=\tilde{a}_{n}$, and satisfying that 
$$ m(\iota_{k}([0,T_{k}]\times[0,1]\setminus E_{k}))/m(Q_{k})=a_{n}.$$

Therefore, $m(A_{n})=a_{n} m(A_{n-1})$, which completes the proof. \end{proof}

\begin{remark}
Throughout Section \ref{section:construction2}, Section \ref{section:shapes} and Section \ref{section:positive_area} we have assumed, to simplify the notation, that the $R_{k}$ provided by Proposition \ref{proposition:phi} at step $n$ in the proof of Theorem \ref{theorem:log-singular_welding} are the same. As we have emphasized, this is not true and changing the arguments accordingly requires some obvious modifications. However, when proving Proposition \ref{proposition:area_left}, we could have also seen that we can match each $R_{k}$ and obtain the desired leftover area (which follows from the arguments there). Therefore, this simplification of the notation does not alter the validity of the result.
\end{remark}


\section{Flexible curves of dimension $1$}\label{section:dim_1}

In this section we prove Theorem \ref{theorem:non-injectivity} for $s=1$. More precisely:

\begin{theorem}[$1$-dimensional flexible curves]\label{theorem:dim_1}
Let $h\colon\Sone\to\Sone$ be an orientation-preserving log-singular circle homeomorphism. Then there is a flexible curve $\Gamma$ with welding $h$ that has Hausdorff dimension $1$.
\end{theorem}

\begin{remark}
	In fact, our arguments in this section will justify that if $m$ is a gauge function with $m(t)=o(t)$ and $h\colon\Sone\to\Sone$ is a log-singular circle homemorphism, then there is a flexible curve $\Gamma$ with welding $h$ so that $H_{m}(\Gamma)=0$.
\end{remark}

We will proceed as in the proof of Theorem \ref{theorem:log-singular_welding}, that is, we consider conformal maps $f_{0}\colon\D\to\Omega_{0}$ and $g_{0}\colon\D^{*}\to\Omega_{0}^{*}$, where $\Omega_{0},\Omega_{0}^{*}$ are smooth Jordan domains with disjoint closures. However, after we construct the maps $\tilde{f}_{1}\colon\D\to\tilde{\Omega}_{1}$ and $\tilde{g}_{1}\colon\D\to\tilde{\Omega}_{1}^{*}$ from Theorem \ref{theorem:log-singular_welding}, we consider the quasiconformal map $H_{1}\colon\C\to\C$ fixing $0,1,\infty$ so that $f_{1}=H_{1}\circ \tilde{f}_{1}$, $g_{1}=H_{1}\circ\tilde{g}_{1}$ are conformal and with $\mu_{H_{1}}=\mu_{1}$ having zero dilatation on $\C\setminus(\tilde{\Omega}_{1}\cup\tilde{\Omega}_{1}^{*})$.

From Theorem \ref{theorem:Ahlfors_formula} we obtain:

\begin{corollary}\label{corollary:small_Beltrami}
	Let $\mu\colon\C\to\D$ be a Beltrami coefficient, that is $\mu$ is a measurable function with $||\mu||_{\infty}\leq k<1$, and let $H\colon\C\to\C$ be the unique quasiconformal map with $\mu_{H}=\mu$ that fixes $0,1,\infty$. If $\mu$ has compact support contained in $\D(0,R)$, then for $z\in \D(0,R)$ we have
	$$ |H(z)-z|\leq C||\mu||_{\infty}.$$
\end{corollary}

Corollary \ref{corollary:small_Beltrami} can also be proved by using the results from Section A in \cite[Chapter V]{AhlforsQC} together with \cite[Chapter V.B, Theorem 1]{AhlforsQC}.

We will use Corollary \ref{corollary:small_Beltrami} applied to the Beltrami coefficient $\mu_{1}$ given as before. It follows from Proposition \ref{proposition:interpolation} (Proposition \ref{proposition:admissible_shape}) and Proposition \ref{proposition:beta_2}, which we used to prove Theorem \ref{theorem:psi}, that we can guarantee we have $||\mu_{1}||_{\infty}$ as small as needed, as well as $\mu_{1}$ having support of arbitrarily small measure.

\begin{figure}[h]
	\includegraphics[scale=1]{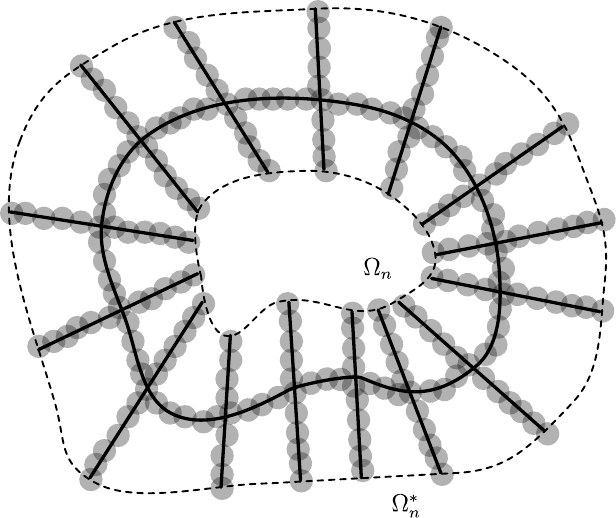}
	\centering
	\caption{Construction in the proof of Theorem \ref{theorem:dim_1}. We choose shapes so that $\C\setminus((\psi_{n}\circ\phi_{n})(\D)\cup(\psi_{n}^{*}\circ\phi_{n}^{*})(\D^{*}))$ consists of a finite union of smooth curves, which is represented by the dark thick curves. We can cover those curves with a finite union of disks (represented in light gray).}
	\label{figure:dim1}
\end{figure}

To prove Theorem \ref{theorem:dim_1} we need to repeatedly use Corollary \ref{corollary:small_Beltrami} to make sure that we can obtain the estimates from Theorem \ref{theorem:log-singular_welding}, while also having control of coverings of the corresponding sequence of annuli (which generate the curve).

\begin{proof}
	Let $f_{0},g_{0}$ be conformal as before. Consider the mappings $\phi_{1},\phi_{1}^{*}$ (Proposition \ref{proposition:phi}) and $\psi_{1},\psi_{1}^{*}$ (Theorem \ref{theorem:psi}), where $\phi_{1},\psi_{1}$ correspond to $f_{0}$ and $\phi_{1}^{*},\psi_{1}^{*}$ correspond to $g_{0}$. We choose $\psi_{1},\psi_{1}^{*}$ so that the compact set
	$$ \C\setminus\left( (\psi_{1}\circ\phi_{1})(\D)\cup (\psi_{1}^{*}\circ\phi_{1}^{*})(\D^{*})  \right)$$
	consists of a finite union of smooth curves. See Figure \ref{figure:dim1}. Since smooth curves have Hausdorff dimension $1$, there is a covering by finite number of open disks $\{D^{1}_{j}\}$ of radius $\{r^{1}_{j}\}$ so that $$ \sum_{j} (r^{1}_{j})^{1+1}\leq 1/2.$$ 
	If the dilatation of the mapping $H_{1}$ is small enough, it follows from Corollary \ref{corollary:small_Beltrami} that we can consider $t$ close enough to $1$ in Lemma \ref{lemma:control_approx} so that the annulus
	$$ A_{1}=A_{1}(t)=\C\setminus\left( (\psi_{1}\circ\phi_{1})(t\D)\cup (\psi_{1}^{*}\circ\phi_{1}^{*})(\D^{*}/t)  \right)$$
	satisfies:
	\begin{enumerate}[label=(\alph*)]
		\item $A_{1}\subset\cup_{j} D^{1}_{j}$.
		\item $H_{1}(A_{1})\subset\cup_{j} D^{1}_{j}$.
		\item The mappings $f_{1}(z)=(H_{1}\circ\psi_{1}\circ\phi_{1})(tz)$ and $g_{1}(z)=(H_{1}\circ\psi_{1}^{*}\circ\phi_{1}^{*})(z/t)$ satisfy the estimates (i) and (ii) from Theorem \ref{theorem:log-singular_welding}.
	\end{enumerate}
	Define $\Omega_{1}=f_{1}(\D)$ and $\Omega_{1}^{*}=g_{1}(\D^{*})$. 
	
	In an analogous way, for $n>1$ we can define conformal maps $f_{n}\colon\D\to\Omega_{n}$, $g_{n}\colon\D^{*}\to\Omega_{n}^{*}$ that satisfy with uniform constants:
	\begin{enumerate}[label=(\roman*)]
		\item $|f_{n}(x)-g_{n}(h(x))|\lesssim (2/3)^{n}$ for $x\in\Sone$.
		\item $|f_{n}(x)-f_{n-1}(x)|\lesssim (2/3)^{n}$, $|g_{n}(x)-g_{n-1}(x)|\lesssim (2/3)^{n}$ for $x\in\Sone$.
	\end{enumerate}
	And so that there is a finite number of disks $D_{j}^{n}$, each of radius $r_{j}^{n}$, that satisfies
	\begin{enumerate}[label=(\alph*)]
		\item $A_{n}=\C\setminus\left( f_{n}(\D)\cup g_{n}(\D^{*})\right) \subset\cup_{j} D^{n}_{j}.$
		\item $ \sum_{j} (r_{j}^{n})^{1+1/n}\leq 2^{-n}$.
		\item $\cup D_{j}^{n}\subset \cup D_{j}^{n-1}$.
	\end{enumerate}
	It follows that the sequence of conformal maps $\{f_{n}\}$, $\{g_{n}\}$ converge to conformal maps $f\colon\D\to\Omega_{\infty}$ and $g\colon\D^{*}\to\Omega_{\infty}^{*}$, where $\Gamma=\partial\Omega_{\infty}=\partial\Omega_{\infty}^{*}$ is a Jordan curve with conformal welding $h$ and Hausdorff dimension $1$.
\end{proof}


\section{Flexible curves of dimension $s\in(1,2)$}\label{section:dim_s}

In this section we complete the proof of Theorem \ref{theorem:non-injectivity}, which we write again for the reader's convenience.

\begin{theorem}\label{theorem:dim_alpha}
Let $h\colon\Sone\to\Sone$ be an orientation-preserving log-singular circle homeomorphism and $s\in[1,2]$. Then there is a flexible curve $\gamma$ with welding $h$ so that $\mathcal{H}dim(\gamma)=s$.
\end{theorem}

Theorem \ref{theorem:positive_area} and Theorem \ref{theorem:dim_1} proves Theorem \ref{theorem:dim_alpha} for $s=2$ and $s=1$. Hence it remains to prove the case $s\in(1,2)$.

The constructions that we made over the past sections use quasiconformal mappings. It is well-known that quasiconformal maps are bi-Holder (see \cite[page 30]{AhlforsQC} or \cite[Theorem 3.10.2]{AstalaBook}) and not in general bi-Lipschitz (it is easy to generate a quasiconformal map of the plane so that in a dense set of $\C$ the derivative exists and vanishes). It is also well-known that quasiconformal maps change Hausdorff dimension (see \cite{Astala-distortion, MR324028, MR1101223, Smirnov_quasicircles}). In \cite[Corollary 1.3]{Astala-distortion} Astala proved that if $f\colon\Omega\to f(\Omega)$ is $K$-quasiconformal and $E\subset\Omega$ is compact, then
$$ \mathcal{H}\mathrm{dim}(f(E))\leq\frac{2K \mathcal{H}\mathrm{dim}(E)}{2+(K-1)\mathcal{H}\mathrm{dim}(E)}.$$ 
If $K\simeq 1$, then $||\mu_{f}||_{\infty}\simeq0$ and $K\leq(1+C||\mu_{f}||_{\infty})$ for $C<\infty$. Hence,
\begin{equation}\label{equation:dimension_distortion}
\frac{\mathcal{H}\mathrm{dim}(E)}{1+C||\mu_{f}||_{\infty}}\leq\mathcal{H}\mathrm{dim}(f(E))\leq(1+C||\mu_{f}||_{\infty})\mathcal{H}\mathrm{dim}(E).
\end{equation}

From (\ref{equation:dimension_distortion}) we obtain.

\begin{lemma}\label{lemma:same_hdim}
Let $E\subset\Omega$ be compact and $f\colon\Omega\to f(\Omega)$ be $K$-quasiconformal. If for every $\eps>0$ there exists a neighborhood $U_{\eps}\supset E$ contained in $\Omega$ so that $f_{\eps}=f_{| U_{\eps}}$ has dilatation $||\mu_{f_{\eps}}||_{\infty}\leq\eps$, then  $\mathcal{H}\mathrm{dim}(E)=\mathcal{H}\mathrm{dim}(f(E))$.
\end{lemma}

In our case, in Theorem \ref{theorem:log-singular_welding} we proved that given a log-singular circle homeomorphism $h\colon\Sone\to\Sone$ and a sequence $\{\eps_{n}\}_{n}\in(0,\infty)$ so that if we set $K_{n}=\prod_{j=1}^{n}(1+\eps_{j})$, then $K=\lim_{n} K_{n}=\prod_{n}(1+\eps_{n})<\infty$. There exists sequences $f_{n}\colon\D\to\Omega_{n}$ and $g_{n}\colon\D^{*}\to\Omega_{n}^{*}$ of $K_{n}$-quasiconformal mappings, where $\Omega_{n},\Omega_{n}^{*}$ are Jordan domains with disjoint closures (with $\infty\in\Omega_{n}^{*}$). 

The sequences of maps $f_{n},g_{n}$ converge respectively to $K$-quasiconformal mappings $f\colon\D\to\Omega_{\infty}$, $g\colon\D^{*}\to\Omega_{\infty}^{*}$, that map onto the complementary components of a Jordan curve $\Gamma$ so that $h=g^{-1}\circ f$. If we define the annulus $$A_{n}=\C\setminus\overline{\Omega_{n}\cup\Omega_{n}^{*}},$$ then $\Gamma=\limsup_{n}A_{n}$.

By the measurable Riemann mapping theorem (Theorem \ref{theorem:MRMT}) there exists a $K$-quasiconformal map $H\colon\C\to\C$ so that $H\circ f$ and $H\circ g$ are conformal (where we set the dilatation over $\Gamma$ to be $0$). Hence the Jordan curve $H(\Gamma)$ has conformal welding $h\colon\Sone\to\Sone$. Our goal is to use (\ref{equation:dimension_distortion}) together with Lemma \ref{lemma:same_hdim} to prove that, under some conditions, the curves $\Gamma$ (\textit{model} flexible curve) and $H(\Gamma)$ (actual flexible curve) have the same Hausdorff dimension when $s\in(1,2)$.

To do so, we rely on using that Proposition \ref{proposition:interpolation} (Proposition \ref{proposition:admissible_shape}) and Proposition \ref{proposition:beta_2} also quantify where we \textit{add dilatation} at each step of our inductive construction. We recall the notation that we need. In Theorem \ref{theorem:psi} we have quadrilaterals $Q_{k}^{n}$ (we emphasize the step $n$ here), and a conformal map $$ (\iota_{k}^{n})^{-1}\colon Q_{k}^{n}\to[0,T_{k}^{n}]\times[0,1]. $$ By Proposition \ref{proposition:admissible_shape}, there are $\eps_{n}$-admissible shapes provided there is $M_{n}=M(\eps_{n})$ so that $2 M_{n}\leq T_{k}^{n}$. After we choose a shape that satisfies the conditions from Proposition \ref{proposition:admissible_shape}, the dilatation of the map $\alpha_{k}^{n}$ is contained in $[0,M_{n}]\times[0,1]$. Similarly, the dilatation of the map $\beta_{k}^{n}$ is contained in $i_{k}^{n}([0,M_{n}]\times[0,1])$.

Hence, following the arguments in the proof of Theorem \ref{theorem:psi} (see also Figure \ref{figure:diagram}), the map $\psi_{n}$ has non-zero dilatation within $A_{n}$ in 
$$ \cup_{k}\iota_{k}^{n}([0,M_{n}]\times[0,1]).$$ 
Note that the quotient $M_{n}/T_{k}^{n}$ can be made as small as necessary.
Observe that $\psi_{n}(\Sone)\cap\partial A_{n}$ consists of a finite number of points. Consider also the maps $\phi_{n}^{*},\psi_{n}^{*}$ when using Proposition \ref{proposition:phi} and Theorem \ref{theorem:psi} with $g_{n}$.

When we use Lemma \ref{lemma:control_approx} for $t<1$, we obtain an annulus $A_{n+1}(t)$ given by
$$ A_{n+1}(t)\coloneqq\C\setminus\overline{\left((\psi_{n}\circ\phi_{n})(t\D)\cup \psi_{n}^{*}\circ(\phi_{n}^{*})(\D/t)\right)}.$$ Within $A_{n+1}(t)$ the mapping $f_{n+1}^{-1}$ has dilatation supported in the compact set $$ B_{n+1}(t)= \overline{A_{n+1}(t)}\cap\left(\overline{\Omega_{n}}\cup\bigcup_{k}\iota_{k}^{n}([0,M_{n}]\times[0,1])\right).$$  We define $A_{n+1}=A_{n+1}(t_{n+1})$ for some $t_{n+1}$ close $1$ (so that the conditions in Theorem \ref{theorem:log-singular_welding} are satisfied). We also define $B_{n+1}=B_{n+1}(t_{n+1})$. From Lemma \ref{lemma:same_hdim} we obtain:

\begin{corollary}\label{corollary:same_dim}
In the same conditions as before, if $H\colon\C\to\C$ is the $K$-quasiconformal map so that $H\circ f$ and $H\circ g$ are conformal with $(\mu_{H})_{|\Gamma}=0$ a.e. Then, if $E$ compact satisfies 
$$E\subset A_{n+1}\setminus B_{n+1}$$ for all $n$, we have  $\mathcal{H}\mathrm{dim}(E)=\mathcal{H}\mathrm{dim}(H(E))$.
\end{corollary}

\begin{remark}
	Observe that as $t\to1$, $B_{n+1}(t)$ converges, in the Hausdorff metric, to a finite union of smooth curves. Hence there exists a sequence $\{t_{n}\}_{n}\subset(0,1)$ so that the set $B=\limsup_{n} B_{n}$ has Hausdorff dimension $1$.
\end{remark}

We conclude the following.

\begin{corollary}\label{corollary:same_dim_curve}
In the same conditions as before, if there is $E\subset\Gamma$ compact so that $E\subset A_{n+1}\setminus B_{n+1}$, $\mathcal{H}\mathrm{dim}(\Gamma\setminus E)=1$ and $\mathcal{H}\mathrm{dim}(E)>1$. Then if $K=\lim_{n} K_{n}$ is close enough to $1$, we have $$ \mathcal{H}\mathrm{dim}(\Gamma)=\mathcal{H}\mathrm{dim}(H(\Gamma)).$$ \end{corollary}
\begin{proof}
From Corollary \ref{corollary:same_dim} we obtain $\mathcal{H}\mathrm{dim}(E)=\mathcal{H}\mathrm{dim}(H(E))$. If we take $K$ close enough to $1$ so that 
$$ \frac{2 K}{2+(K-1)}\leq K\leq \mathcal{H}\mathrm{dim}(E),$$
we have, 
\begin{align*}
\mathcal{H}\mathrm{dim}(H(\Gamma))&=\max\left\{\mathcal{H}\mathrm{dim}(H(E)),\mathcal{H}\mathrm{dim}(H(\Gamma\setminus E))\right\} \\ 
&=\mathcal{H}\mathrm{dim}(H(E))=\mathcal{H}\mathrm{dim}(E)=\mathcal{H}\mathrm{dim}(\Gamma).\qedhere
\end{align*} \end{proof}

In the conditions of Corollary \ref{corollary:same_dim_curve} justifies that the \textit{model} curve from Theorem \ref{theorem:log-singular_welding} and the actual flexible curve (the one after post-composing with the quasiconformal map $H$ as before) have the same Hausdorff dimension when $s>1$. We will see that for each $s\in(1,2]$ we can find such Jordan curve.

The strategy that we will use to construct the curve is the following outline from \cite[Section 4.12]{Mattila:GMT}, which we condense into the following proposition.

\begin{proposition}\label{proposition:mattila}
Let $s\in(0,2]$ and suppose that for $k\in\N$ we have compact sets $E_{i_{1},\ldots, i_{k}}$, for $i_{j}=1,\ldots,m_{j}$, such that if $d(\cdot)$ denotes the diameter of a set, the following conditions are satisfied:
\begin{enumerate}[label=(\roman*)]
	\item $E_{i_{1},\ldots,i_{k},i_{k+1}}\subset E_{i_{1},\ldots,i_{k}}$.
	\item $d_{k}=\max_{i_{1}\cdots i_{k}} d(E_{i_{1},\ldots,i_{k}})\to 0$ as $k\to\infty$. 
	\item $\sum_{j=1}^{m_{k+1}} d(E_{i_{1},\ldots,i_{k},j})^{s}=d(E_{i_{1},\ldots,i_{k}})^{s}$.
	\item For any ball $B$ with $d(B)\geq d_{k}$, there is a positive constant $c>0$ so that
	$$ \sum_{B\cap E_{i_{1},\ldots,i_{k}}} d(E_{i_{1},\ldots,i_{k}})^{s}\leq c d(B)^{s}.$$
\end{enumerate} 
Then, if we define $$ E=\bigcap_{k\geq1}\bigcup_{i_{1}\cdots i_{k}} E_{i_{1},\ldots, i_{k}},$$ we have $0<H_{s}\left(E\right)<\infty.$ Hence $E$ has Hausdorff dimension $s$.
\end{proposition}

Lemma \ref{lemma:split_s} below guarantees that when we construct the model flexible curve $\Gamma$ we can achieve (iii) in Proposition \ref{proposition:mattila}.

\begin{figure}[h]
	\includegraphics[scale=1]{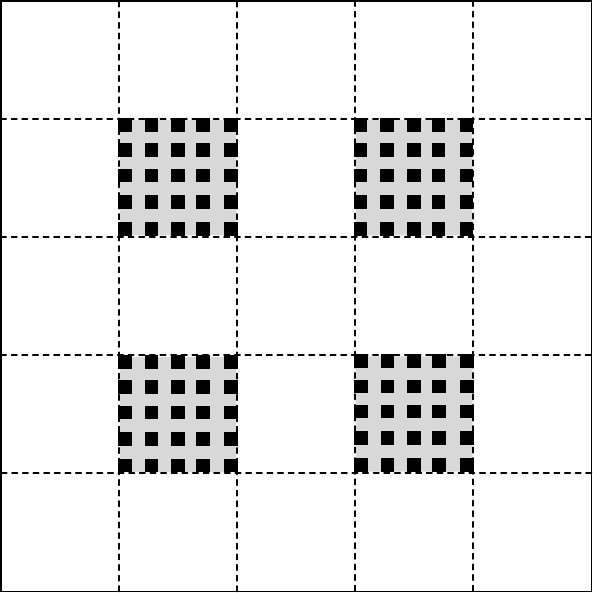}
	\centering
	\caption{Finding $s$-additive squares (the ones that are dark colored) that are well-separated as in Lemma \ref{lemma:split_s} and Lemma \ref{lemma:separation_s}.}
	\label{figure:s-additive_squares}
\end{figure}

\begin{lemma}[Finding $s$-additive squares]\label{lemma:split_s}
	Let $Q=[0,1]^{2}$ and $s\in(1,2)$. For $m\in\{1,2,3,4\}$ fix the squares $R^{m}\subset Q$ with disjoint closures and side-length $1/5$ so that $d(\partial Q, R^{m})=1/5$ and for $i\not=j$, $d(R^{i},R^{j})\geq1/5$. 
	Then there exists $n=n(s)\in\N$, $x=x(s)\in(0,1)$ and a sub-collection of closed squares $\{Q^{m}_{j}\}_{j=1}^{n}\subset Q$, where $m\in\{1,2,3,4\}$, with disjoint interiors so that:
	\begin{enumerate}[label=(\roman*)]
		\item Each $Q_{j}^{m}\subset R^{m}$ is a square with side-length $x$.
		\item We have $1=4n(\sqrt{2}x)^{s}$, that is, the squares $Q_{j}^{m}$ satisfy
		$$ 1^{s}=\sum_{j,m} d(Q_{j}^{m})^{s}=4n 2^{s/2} x^{s}.$$
	\end{enumerate}
	Moreover, we can choose the squares $\{Q_{j}^{m}\}$ to be uniformly distributed in each $R^{m}$. That is, if we define the ratio $P(s)$ between $n=n(s)$ and the total number of squares in $R^{m}$, $$ P(s)\coloneqq \left(\frac{25}{4\cdot 2^{s/2}}\right) x^{2-s},$$
	then for each disk $D(z,r)$ with $z\in R^{m}$ and $r>0$, the number of squares $Q_{j}^{m}$ that intersects $D(z,r)$ is comparable to 
	\begin{equation}\label{equation:expected_number} P(s) \cdot \frac{r^{2}}{x^{2}}.\end{equation}
	In particular, the distance from $Q_{j}^{m}$ to its nearest neighbor is comparable to $x^{s/2}$.
\end{lemma}
\begin{proof}
Assume that $n\in\N$, $x\in(0,1)$ are taken so that $1=4 n 2^{s/2} x^{s}$. For condition (i) to be satisfied we need
$$ n x^{2} = \frac{x^{2-s}}{4 \cdot 2^{s/2}}\leq 1/25.$$
Since $s\in(1,2)$, it follows that there is an interval of $x$ for which this holds (at the expense of making $n$ large). Thus there is a pair $n=n(s)\in\N$ and $x=x(s)\in(0,1)$ for which (ii) holds. 

To prove the last part, we choose the $Q_{j}^{m}$ to be uniformly distributed within $R^{m}$ (see Figure \ref{figure:s-additive_squares}), as defined in the statement of the Lemma. If $d=d(s)$ is the distance from $Q_{j}^{m}$ to its nearest neighbor, then
$$ \frac{9 x^{2}}{(3 x + 2 d)^{2}}\simeq \left(\frac{25}{4\cdot 2^{s/2}}\right) x^{2-s}.$$
Hence, $x^{s}\simeq (3x +2d)^{2}$. This implies $d\simeq x^{s/2}.$\end{proof}

Observe that if instead of having $[0,1]^{2}$ we have any square $Q$, then there is an affine transformation $L$ that maps $[0,1]^{2}$ to $Q$. Under such affine transformation the value $x=x(s)$ in Lemma \ref{lemma:split_s} gets scaled according to $L$, but the number $n=n(s)$ does not change. In particular, the squares obtained in Lemma \ref{lemma:split_s} under the affine transformation $L$ get mapped to squares satisfying an analogous statement, with condition (ii) being preserved.

As we mentioned before, Lemma \ref{lemma:split_s} will guarantee that condition (iii) from Proposition \ref{proposition:mattila}. The goal now is to see that condition (iv) in Proposition \ref{proposition:mattila} is also satisfied when we take the squares from Lemma \ref{lemma:split_s}. See Figure \ref{figure:s-additive_squares} for an illustration.

\begin{lemma}[$s$-additive squares with large separation]\label{lemma:separation_s}
Let $Q=[0,1]^{2}$ and the squares $R^{m}$ and $Q_{j}^{m}$ as in Lemma \ref{lemma:split_s}. Then, for any disk $D$ of radius $r\geq x(s)$, where $x(s)=x$ is the side-length of each $Q_{j}^{m}$, we have
$$ \sum_{D\cap Q_{j}^{m}\not=\emptyset} d(Q_{j}^{m})^{s}\lesssim d(D)^{s}\simeq r^{s}.$$
\end{lemma}
\begin{proof}
	Suppose first that $ x\leq r\leq1/5$ Given a disk of radius $r$ contained in $Q$, the extreme case of the inequality in the Lemma is attained when $D\cap R^{m}\not=\emptyset$. In such a case, for $\tilde{m}\not=m$ we have $D\cap R^{\tilde{m}}=\emptyset$.
	Since the density $P(s)$, as defined in Lemma \ref{lemma:split_s}, of the squares $Q_{j}^{m}$ is comparable to $x^{2-s}$, a disk of radius $r$ as above will intersect at most $ x^{2-s} r^{2}/x^{2}= r^{2}/x^{s} $ squares. By using $d(Q_{j}^{m})\simeq x$, we obtain
	$$ \sum_{D\cap Q_{j}^{m}\not=\emptyset} d(Q_{j}^{m})^{s}\lesssim x^{s} r^{2}/x^{s}=r^{2}\leq r^{s}.$$
	This condition is also implied by knowing that the distance from $Q_{j}^{m}$ to its nearest neighbor is comparable to $x^{s/2}$. That is because if we have $m^{2}$ squares in a square of side-length $r$, then $r \simeq m x+m x^{s/2}\simeq m x^{s/2} .$
	
	For $r\geq 1/5$, the extreme case is when all the $Q_{j}^{m}$ are contained in $D$, but in that case,
	$$ \sum_{D\cap Q_{j}^{m}\not=\emptyset} d(Q_{j}^{m})^{s} = 1 \lesssim r^{s}.$$
	This completes the proof of the Lemma. \end{proof}

\begin{figure}[h]
	\includegraphics[scale=1]{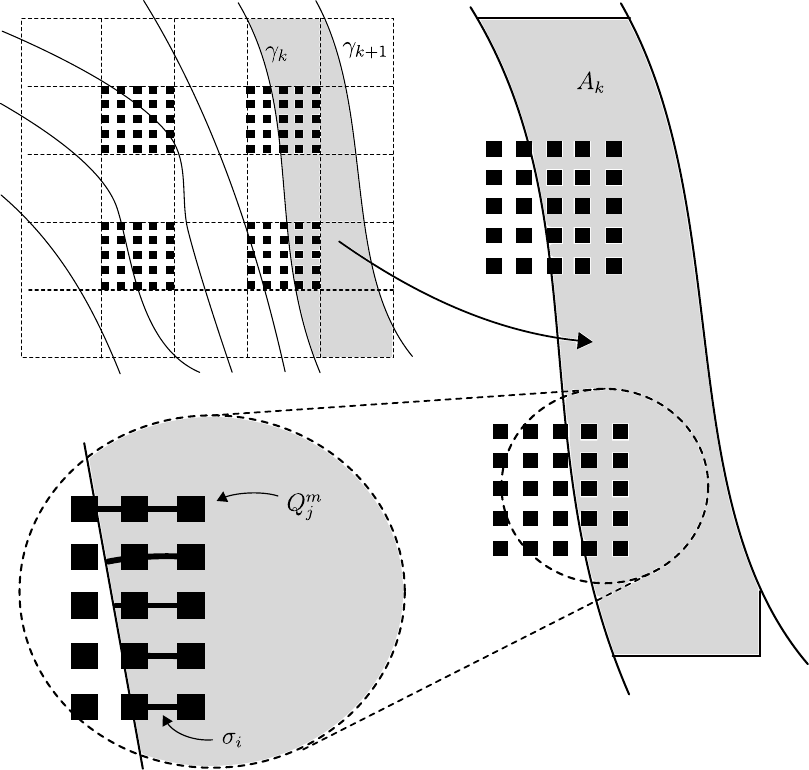}
	\centering
	\caption{In Proposition \ref{proposition:connecting_s} we realize the well-separated $s$-additive squares $Q_{j}^{m}$ from Lemma \ref{lemma:separation_s} as part of the complement of a shape, as in Section \ref{section:shapes}. The squares $Q_{j}^{m}$ are joined by \textit{thin} quadrilaterals $\sigma_{i}$}
	\label{figure:connecting_s-additive_squares}
\end{figure}
	
We now prove that the $s$-additive squares found in Lemma \ref{lemma:split_s} and Lemma \ref{lemma:separation_s} can be located within a partition of $[0,1]^{2}$ by smooth curves, contained in the union of the complement of shapes (as in Section \ref{section:shapes}) and so that the conditions in Corollary \ref{corollary:same_dim_curve} will be satisfied. This is represented in Figure \ref{figure:connecting_s-additive_squares}.
	
\begin{proposition}[Connecting $s$-additive well-separated squares]\label{proposition:connecting_s}
Let $s\in(1,2)$, $\delta>0$, $Q=[0,1]^{2}$ and a finite collection of $N$ smooth simple disjoint curves $\{\gamma_{k}\}$ with endpoints on $\partial Q$ so that:
\begin{enumerate}[label=(\alph*)]
	\item $\gamma_{k}$ is isotopic to $\gamma_{\tilde{k}}$ for any $k,\tilde{k}$.
	\item Each $\gamma_{k}$ intersects the interior of $Q$ and does not intersect the same side of $Q$ twice. 
	\item The $\{\gamma_{k}\}$ are ordered, that is, for $0\leq k<N$ there is a unique component $A_{k}$ of $Q\setminus\{\gamma_{k},\gamma_{k+1}\}$ that does not contain any other $\gamma_{j}$.
	\item $A_{k}$ is a quadrilateral.
\end{enumerate}	
	Then, if we consider the squares $Q_{j}^{m}\subset R^{m}$ from Lemma \ref{lemma:split_s}, which have side-length $x=x(s)$. Then for $x$ small enough there is a finite collection of quadrilaterals  $\{\sigma_{i}\}$ that satisfy:
	\begin{enumerate}[label=(\roman*)]
		\item $\sigma_{i}\subset \overline{A_{k}}$ and it intersects at least one $Q_{j}^{m}$ on one of the sides of $\sigma_{i}$.
		\item Each $\sigma_{i}$ intersects at most two different $Q_{j}^{m}\subset A_{k}$.
		\item The set, $$ \left(\bigcup_{Q_{j}^{m}\cap A_{k}\not=\emptyset} Q_{j}^{m}\right) \cup \bigcup_{\sigma_{i}\cap A_{k}\not=\emptyset} \sigma_{i} $$ bounds a shape, as in Section \ref{section:shapes}, with respect to the quadrilateral $A_{k}$.
		\item There is a covering $\{D_{j}\}$, where each $D_{j}$ is a disk of radius $r_{j}$, of $\cup \sigma_{i}$ so that, $$ \sum_{j} r_{j}^{s}\leq\delta.$$
	\end{enumerate}
	Moreover, given $\eps>0$  the shapes in (iii) can be taken to be $\eps$-admissible provided $d(\gamma_{k}, \gamma_{k+1})$ is small enough.
\end{proposition}
\begin{proof}
Take $x=x(s)$ from Lemma \ref{lemma:split_s} small enough so that no $Q_{j}^{m}$ intersects two different $\gamma_{k}$. Now, given $A_{k}$ and the set of squares $Q_{j}^{m}$ so that their interior intersects with $A_{k}$, join them by considering a finite number of smooth disjoint simple curves $\beta_{i}$ so that,
\begin{enumerate}[label=(\alph*)]
	\item $\beta_{i}$ crosses at least one $Q_{j}^{m}$, and away from the corner of $Q_{j}^{m}$.
	\item Each $\beta_{i}$ only intersects one side of $A_{k}$.
	\item No two different $\beta_{i}, \beta_{l}$ intersect the same $Q_{j}^{m}$.
	\item $d(\beta_{i},\partial Q)\geq1/5$.
\end{enumerate}

 These curves are represented in Figure \ref{figure:connecting_s-additive_squares}.

We can cover the union of the $\beta_{i}$ by finitely many open disks $\{D_{j}\}$, with $D_{j}$ of radius $r_{j}$, so that if $s,\delta$ are as in the statement of the lemma, we have:
 $$ \sum_{j} r_{j}^{s}\leq\delta.$$
Since the disks $D_{j}$ are open, there is a thickening $\sigma_{i}$ of the curves that satisfies conditions (i), (ii), (iii) and (iv). 
 
 Finally, each $A_{k}$ can be conformally mapped to a rectangle, and the set $$A_{k}\cap\{z\in Q\colon d(z,\partial Q)\leq 1/4\}\subset A_{k}$$ bounds a quadrilateral $B_{k}$. Since $R^{m}$ satisfy $d(R^{m},\partial Q)=1/5$, it follows that as $d(\gamma_{k},\gamma_{k+1})$ decreases, the modulus of the quadrilateral $B_{k}$ degenerates. Hence, by Proposition \ref{proposition:admissible_shape}, for $d(\gamma_{k},\gamma_{k+1})$ small enough every shape within $A_{k}$ that contains $B_{k}$ is $\eps$-admissible.
\end{proof}

We are now ready to prove Theorem \ref{theorem:dim_alpha} for $s\in(1,2)$.

\begin{proof}[Proof of Theorem \ref{theorem:dim_alpha}]
The case $s=2$ is Theorem \ref{theorem:positive_area} and the case $s=1$ is Theorem \ref{theorem:dim_1}, hence we suppose $s\in(1,2)$.

We need to see that the sequence of annuli $A_{n}$ in the proof of Theorem \ref{theorem:log-singular_welding} can be taken in such a way that the Jordan curve $\Gamma=\limsup A_{n}$ has Hausdorff dimension $s$ and satisfies the conditions of Corollary \ref{corollary:same_dim_curve}.

To do so, start with $A_{0}=\{1<|z|<4\}$ and consider the quadrilaterals $Q_{k}$ as in Theorem \ref{theorem:psi}. Take a square of side-length $1$ that is contained in the union of $Q_{k}$. Consider the conformal map of $\C$ that sends $\tilde{Q}_{0}$ to $Q=[0,1]^{2}$. We apply Proposition \ref{proposition:connecting_s} with $s$ and $\delta=1/2^{0}$. If the quadrilaterals $Q_{k}$ are thin enough, then we can obtain $\eps_{1}$-admissible shapes by Proposition \ref{proposition:connecting_s} and so that $\eps_{1}$-admissible shapes exist on the remaining $Q_{k}$. This generates the first iteration of sets $\tilde{E}_{i}$ as in Proposition \ref{proposition:mattila}.

We take the shapes in the remaining $Q_{k}$ in such a way that the complement of the shape within any two adjacent $Q_{k}$'s can be covered by open disks of radius $r_{j}$ that satisfy
$$ \sum r_{j}^{s}\leq (1/2^{0})/ (N+2),$$
where $N$ is the total number of $Q_{k}$'s where the shape from Proposition \ref{proposition:connecting_s} has not been added. This defines a set $\tilde{A_{1}}$ so that $\tilde{A}_{1}\setminus \cup_{i} \tilde{E}_{i}$ can be covered by a finite number of disks $D_{j}$ of radius $r_{j}$ satisfying
$$ \sum r_{j}^{s}\leq (1/2^{0}) (N+1)/(N+2).$$

As in the proof of Theorem \ref{theorem:log-singular_welding}, we now need to use Lemma \ref{lemma:control_approx}. When we do so, the sizes will change, but if $t$ in Lemma \ref{lemma:control_approx} is small enough we can guarantee that for a fixed $a_{1}>1$ we obtain an annulus $A_{1}$ and sets $E_{i}\supset \tilde{E}_{i}$ contained in $A_{1}\cap \tilde{Q}_{0}$ so that
$$ \sum d(E_{i})^{s}=a_{1} d(\tilde{Q}_{0})^{s},$$
with $A_{1}\setminus\cup_{i} E_{i}$ admitting a covering by a finite number of disks $D_{j}$ of radius $r_{j}$ that satisfy $$ \sum_{j}r_{j}^{s}\leq 1/2^{0},$$
and so that for any disk $D$ with $d(D)\geq d_{k}$ as in Proposition \ref{proposition:mattila} we have,
$$ \sum_{D\cap E_{i}\not=\emptyset} d(E_{i})^{s}\lesssim a_{1} d(D)^{s}.$$

Suppose that at step $n$ we have obtained a sequence of annuli $\{A_{k}\}_{k\leq n}$ and sets $E_{i_{1},\ldots, i_{n}}$ so that for $k<n$ we have
\begin{equation}\label{equation:Hdim_1} \sum_{j} d(E_{i_{1},\ldots, i_{k},j})^{s}=d(E_{i_{1},\ldots,i_{k}})^{s}\prod_{j=1}^{k} a_{j} \end{equation}
where $a_{j}>1$. For $k\leq n$ the set $A_{k}\setminus\left(\cup E_{i_{1},\ldots,i_{k}}\right)$ can be covered by a finite number of disks $D_{j}$ of radius $r_{j}$ satisfying 
\begin{equation}\label{equation:Hdim_2} \sum (r_{j})^{1+(s-1)/k}\leq 2^{-k},\end{equation} 
and such that for any disk $D$ with radius $d(D)\geq d_{k}$ as in Proposition \ref{proposition:mattila}, we have
\begin{equation}\label{equation:Hdim_3} \sum_{D\cap E_{i_{1},\ldots, i_{k}}\not=\emptyset} d(E_{i_{1},\ldots, i_{k}})^{s}\lesssim \left(\prod_{j=1}^{k} a_{j}\right) d(D)^{s}. \end{equation}

We now use Proposition \ref{proposition:connecting_s} with $s$ and $\delta=1/2^{n+1}$ to each one of the sets $E_{i_{1},\ldots, i_{n}}$. Applying Lemma \ref{lemma:control_approx} for some $a_{n}>1$ as before generates an annulus $A_{n+1}\subset A_{n}$ and sets $E_{i_{1},\cdots, i_{n+1}}$ that satisfy the equations (\ref{equation:Hdim_1}), (\ref{equation:Hdim_2}) and (\ref{equation:Hdim_3}).

If we take the sequence $\{a_{n}\}$ so that $\prod a_{n}\leq 2$, it follows from a minor modification in the proof of Proposition \ref{proposition:mattila} that the set 
$$ E = \bigcap_{k} \bigcup_{i_{1},\ldots, i_{k}} E_{i_{1},\ldots, i_{k}}$$
has Hausdorff dimension $s$. Moreover, (\ref{equation:Hdim_2}) guarantees that if $\Gamma=\cap A_{n}$, then $$ H_{\tilde{s}}(\Gamma\setminus E)=0$$ for all $\tilde{s}\in(1,s)$. Hence $\Gamma\setminus E$ has Hausdorff dimension $1$ and $\Gamma$ has Hausdorff dimension $s$. 

If when we construct the set $E$ we also guarantee that $A_{n+1}\setminus B_{n+1}$, then Corollary \ref{corollary:same_dim_curve} the corresponding flexible curve also has Hausdorff dimension $s$.\end{proof}

\begin{remark}
It is worthwhile mentioning that a different way to prove that the Hausdorff dimension of $\Gamma$ and $H(\Gamma)$ are equal would be by proving that $H$ is bi-Lipschitz on $\Gamma$. However, characterizing which Beltrami coefficients correspond to bi-Lipschitz mappings remains an open problem. It is known that under some conditions the corresponding quasiconformal map is bi-Lipschitz on a subset of our domain. For example, see \cite[Lemma 15.1]{Albrecht-Bishop}.
\end{remark}

\begin{remark}
A modification of the arguments in this section also yields that given a log-singular circle homeomorphism $h$, there is a flexible curve $\gamma$ with welding $h$ that has Hausdorff dimension $2$ but zero area.
\end{remark}



\bibliographystyle{amsalpha}
\bibliography{references}

\end{document}